\documentclass[11pt, draft]{amsart}
\usepackage{amssymb, amstext, amscd, amsmath, amssymb}
\usepackage{mathtools, xypic, paralist, color, dsfont, rotating}
\usepackage{verbatim}
\usepackage{enumerate,enumitem}
\usepackage{float}
\usepackage{setspace}

\usepackage{tikz}

\floatstyle{boxed} 
\restylefloat{figure}

\numberwithin{equation}{section}
\usepackage[a4paper]{geometry}
\usepackage{quoting}
\quotingsetup{vskip=.1in}
\quotingsetup{leftmargin=.17in}
\quotingsetup{rightmargin=.17in}

\let\OLDthebibliography\thebibliography
\renewcommand\thebibliography[1]{
  \OLDthebibliography{#1}
  \setlength{\parskip}{0pt}
  \setlength{\itemsep}{2pt plus 0.5ex}
}

\makeatletter
\def\@cite#1#2{{\m@th\upshape\bfseries%
[{#1\if@tempswa{\m@th\upshape\mdseries, #2}\fi}]}}
\makeatother
\theoremstyle{plain}
\newtheorem{theorem}{Theorem}[section]
\newtheorem{corollary}[theorem]{Corollary}
\newtheorem{proposition}[theorem]{Proposition}

\theoremstyle{definition}
\newtheorem{definition}[theorem]{Definition}

\newtheorem{remark}[theorem]{Remark}

\newtheorem*{acknow}{Acknowledgements}

\theoremstyle{remark}

\mathtoolsset{centercolon}
  \newcommand{\A}{{\mathcal{A}}}
  \newcommand{\B}{{\mathcal{B}}}

  \newcommand{\E}{{\mathcal{E}}}
  
  \newcommand{\G}{{\mathcal{G}}}

  \newcommand{\I}{{\mathcal{I}}}
  \newcommand{\J}{{\mathcal{J}}}

  \newcommand{\M}{{\mathcal{M}}}

\renewcommand{\P}{{\mathcal{P}}}

\renewcommand{\S}{{\mathcal{S}}}
  \newcommand{\T}{{\mathcal{T}}}
  
  \newcommand{\V}{{\mathcal{V}}}

\def\De{\Delta}
\def\de{\delta}

\def\la{\lambda}

\def\om{\omega}
\def\Om{\Omega}
\def\si{\sigma}

\newcommand{\bC}{\mathbb{C}}

\newcommand{\bN}{\mathbb{N}}

\newcommand{\bZ}{\mathbb{Z}}

\newcommand{\fe}{{\mathfrak{e}}}

\newcommand{\ff}{{\mathfrak{f}}}

\newcommand{\fg}{{\mathfrak{g}}}

\newcommand{\fv}{{\mathfrak{v}}}

\newcommand{\Bx}{{\mathbf{x}}}
\newcommand{\By}{{\mathbf{y}}}
\newcommand{\Bz}{{\mathbf{z}}}

\newcommand{\foral}{\text{ for all }}
\newcommand{\qand}{\quad\text{and}\quad}

\newcommand{\qfor}{\quad\text{for}\quad}

\newcommand{\ca}{\mathrm{C}^*}

\newcommand{\cenv}{\mathrm{C}^*_{\textup{env}}}

\newcommand{\ol}{\overline}

\newcommand{\ad}{\operatorname{ad}}

\newcommand{\alg}{\operatorname{alg}}

\newcommand{\env}{\operatorname{\textup{env}}}

\newcommand{\id}{{\operatorname{id}}}

\newcommand{\mt}{\emptyset}

\newcommand{\qd}{\operatorname{qd}}

\newcommand{\scv}{\operatorname{sc}}
\newcommand{\spn}{\operatorname{span}}

\newcommand{\sumoplus}{\operatornamewithlimits{\sum \strut^\oplus}}

\newcommand{\sca}[1]{\left\langle#1\right\rangle} 
\newcommand{\nor}[1]{\left\Vert #1\right\Vert} 
\newcommand{\quo}[2]{{\raisebox{.1em}{$#1$}\left/ \, \raisebox{-.1em}{$#2$}\right.}} 

\begin{document}

\title[Boundary quotient C*-algebras of semigroups]{Boundary quotient C*-algebras of semigroups}

\author[E.T.A. Kakariadis]{Evgenios~T.A.~Kakariadis}
\address{School of Mathematics, Statistics and Physics\\ Newcastle University\\ Newcastle upon Tyne\\ NE1 7RU\\ UK}
\email{evgenios.kakariadis@newcastle.ac.uk}

\author[E.G. Katsoulis]{Elias~G.~Katsoulis}
\address{Department of Mathematics\\ East Carolina University\\ Greenville\\ NC 27858\\USA}
\email{katsoulise@ecu.edu}

\author[M. Laca]{Marcelo Laca}
\address{Department of Mathematics and Statistics\\ University of Victoria\\ Victoria\\ BC\\ Canada}
\email{laca@uvic.ca}

\author[X. Li]{Xin Li}
\address{School of Mathematics and Statistics\\ University of Glasgow\\ University Place\\ Glasgow\\ G12 8QQ\\ UK}
\email{xin.li@glasgow.ac.uk}



\begin{abstract}  
We study two classes of operator algebras associated with a unital subsemigroup $P$ of a discrete group $G$: 
one related to universal structures, and one related to co-universal structures.
First we provide connections between universal C*-algebras that arise variously from isometric representations of $P$ that reflect the space $\mathcal{J}$ of constructible right ideals, from associated Fell bundles, and from induced partial actions.
This includes connections of appropriate quotients with the strong covariance relations in the sense of Sehnem.
We then pass to the reduced representation $\mathrm{C}^*_\lambda(P)$ and we consider the boundary quotient $\partial \mathrm{C}^*_\lambda(P)$ related to the minimal boundary space.
We show that $\partial \mathrm{C}^*_\lambda(P)$ is co-universal in two different classes: (a) with respect to the equivariant constructible isometric representations of $P$; and (b) with respect to the equivariant C*-covers of the reduced nonselfadjoint semigroup algebra $\mathcal{A}(P)$.
If $P$ is an Ore semigroup, or if $G$ acts topologically freely on the minimal boundary space, then $\partial \mathrm{C}^*_\lambda(P)$ coincides with the usual C*-envelope $\mathrm{C}^*_{\text{env}}(\mathcal{A}(P))$ in the sense of Arveson.
This covers total orders, finite type and right-angled Artin monoids, the Thompson monoid, multiplicative semigroups of nonzero algebraic integers, and the $ax+b$-semigroups over integral domains that are not a field.
In particular, we show that $P$ is an Ore semigroup if and only if there exists a canonical $*$-isomorphism from $\partial \mathrm{C}^*_\lambda(P)$, or from $\mathrm{C}^*_{\text{env}}(\mathcal{A}(P))$, onto $\mathrm{C}^*_\lambda(G)$.
If any of the above holds, then $\mathcal{A}(P)$ is shown to be hyperrigid.
\end{abstract}

\maketitle

\section{Introduction} \addtocontents{toc}{\protect\setcounter{tocdepth}{1}}

The use of C*-constructs has been central in the study of geometric and topological objects such as semigroups, graphs, dynamical systems etc.
It goes as far back as the work of Murray and von Neumann in the 1930's and 1940's, and has been a continuous source of inspiration for further developments.
On one hand one obtains a universal object that covers appropriate representations of the object by Hilbertian operators.
On the other hand one can ask for the minimal quotient that still carries a faithful copy of the original data, through a boundary space.
Similar questions appear in the nonselfadjoint context through the seminal work of Arveson \cite{Arv69} and the notion of the \v{S}ilov boundary.
In a recent work of the authors with Dor-On \cite{DKKLL20} it has been shown that C*-boundaries and Arveson's \v{S}ilov boundary interconnect in a rather solid way in the case of product systems.
In this work we continue to investigate boundary quotients in both the C*- and the nonselfadjoint context. 
This is part of a bigger programme that aims to elucidate the relations between boundary operator algebras.

Here we examine several operator algebras associated with a unital subsemigroup $P$ of a group $G$, always assuming that $P$ generates $G$.
This type of operator algebras have been under thorough investigation by many authors and connections have been established with C*-algebras of more general constructs.
The principal examples come from abelian semigroups, starting with $\bN$ in the seminal work of Coburn \cite{Cob67}, and more general totally ordered groups by Douglas \cite{Dou72} and Murphy \cite{Mur87}.
Nica \cite{Nic92} introduced quasi-lattice ordered semigroups and established the study of their semigroup C*-algebras as universal models of the left regular representation on $\ell^2(P)$.
Further motivating examples come from geometric group theory \cite{CL02, CL07, LOS18} or have a number-theoretic origin \cite{CDL13, CEL13, CEL15}. 
We refer the reader to \cite{CELY17} and the references therein.

A major motivation for working with semigroups is that their operator algebras lie in the intersection of several categories of current interest, such as partial dynamical systems, product systems and Fell bundles.
The different realizations bring in existing results from each category for studying their structure.
Conversely, semigroup algebras can act as a medium for cross-pollination of techniques between their wider supercategories, as well as a testing ground for new structural theorems.
As examples we mention the realization: by partial crossed products for computing the K-theory \cite{CEL13}, for computing the KMS-states \cite{CDL13}, or for deducing simplicity \cite{Li17}; and as product systems for nuclearity results \cite{Kak19, Li12}.

The representations of $P$ that we consider are more than just by isometries as the universal isometric C*-algebra $\ca_{\textup{iso}}(P)$ may not be carrying much structure; for example $\ca_{\textup{iso}}(\bN^2)$ fails to be even nuclear.
The point of inspiration is the left regular representation $\ca_\la(P)$ and in this paper we consider isometric representations of $P$ that reflect various aspects of the structure of the constructible right ideals as they are manifested in $\ca_\la(P)$.
The study is carried mainly in five directions: 
%
%
\begin{inparaenum}
\item universal/maximal constructs; 
\item reduced representations; 
\item boundary quotients; 
\item co-universal C*-algebras; and 
\item \v{S}ilov boundaries.
\end{inparaenum}
%
%
An important aspect in (i)--(iv) has been the existence of a coaction by the group $G$ that induces a topological grading.
In the past, connections with (v) had been established for abelian lattices where the coaction is well understood as an action of the compact dual \cite{DFK17, DK20}.
Three natural questions that arise in this context:

\vspace{1pt}

\begin{enumerate}
\item[{\bf Qn.\ 1.}] What is the connection between universal semigroup constructions? \vspace{1pt} 
\item[{\bf Qn.\ 2.}] What is the connection between the induced semigroup boundary quotients? \vspace{1pt} 
\item[{\bf Qn.\ 3.}] Is there a co-universal semigroup C*-algebra and what is its form?
\end{enumerate}

\vspace{1pt}

Here we answer these questions guided by two beacons (see Figure \ref{F:roadmap} for a summary).
First we establish links between different classes of representations of $P$ and the universal C*-algebra $\ca_s(P)$ of the constructible representations.
Secondly we identify the boundary quotient C*-algebra $\partial \ca_\la(P)$ as the co-universal object in both the C*-algebraic and the nonselfajoint algebraic approaches.
Below we provide a detailed discussion of these results.

\subsection{On universal C*-algebras}

Our starting point is the universal C*-algebra $\ca_s(P)$ with respect to contructible semigroup representations in the sense of \cite{Li12}.
This is an equivariant quotient of the universal isometric C*-algebra $\ca_{\textup{iso}}(P)$, that takes into account the space $\J$ of the constructible right ideals in $P$, and thus sits closer to the structure of $\ca_{\la}(P)$.
A second important variant is the boundary quotient $\ca_{\scv}(P)$ of $\ca_{\textup{iso}}(P)$, given by the strongly covariant representations of $P$ as induced by Sehnem's work on product systems \cite{Seh18}.
This universal C*-algebra models the isometric representations that lift automatically to faithful representations on the fixed point algebra.
In Proposition \ref{P:scv ex} we show that strongly covariant representations are constructible, providing the vertical arrow in the following commutative diagram:
{\small
\[
\xymatrix@R=.3cm@C=.5cm{
\ca_{\textup{iso}}(P) \ar[rr] \ar[ddrr] & & \ca_s(P) \ar[rr] \ar[dd] & & \ca_{\la}(P) \\
& & & & \\
& & \ca_{\scv}(P) & &
}.
\]
}

Universal/maximal constructs are easier to work with, as they enjoy exactness with respect to induced ideals.
We use this type of arguments for identifying $\ca_s(P)$ and $\ca_{\scv}(P)$ with C*-algebras associated with Fell bundles and partial actions.
On one hand we consider the induced Fell bundles $\P_s$ inside $\ca_s(P)$, and $\P_{\scv}$ inside $\ca_{\scv}(P)$.
On the other hand we use the fixed point algebra $D_s(P)$ of $\ca_s(P)$ and the minimal $G$-invariant subspace $\partial \Om_P$ of its spectrum to realize them as partial crossed products (see Theorem \ref{T:Fell bdl univ} and Theorem \ref{T:un pcp}).
Namely we have
\[
\ca_s(P) \simeq \ca(\P_s) \simeq D_s(P) \rtimes G
\qand
\ca_{\scv}(P) \simeq \ca(\P_{\scv}) \simeq C(\partial \Om_P) \rtimes G.
\]
In passing, we show that $\ca_s(P) \simeq \ca(\I_V)$ for the inverse semigroup
\[
\I_V := \{V_{p_1}^* V_{q_1} \cdots V_{p_n}^* V_{q_n} \mid p_i, q_i \in P; n \in \bZ_+\}
\]
of $V$-words in $\ca_\la(P)$ (see Theorem \ref{T:inv sgrp}), which improves a previous result of Li \cite{Li13} by removing the Toeplitz condition from $P$.
By using a result of Norling \cite{Nor14} we deduce that $\J$ is independent if and only if $D_s(P) \simeq D_\la(P)$ canonically (see Corollary \ref{C:nor-li}).
These results address {\bf (Qn.\ 1)}.
In general $\ca_\la(\P_s)$ differs from $\ca_\la(P)$.
The algrebraic relations that define the Fell bundle induced in $\ca_\la(P)$ are completely described by Laca and Sehnem in their recent work \cite{LS21}.
Even though we do not rely here on \cite{LS21}, it has motivated us in the final stages of this research to free the setup from the independence condition on $\J$ and work at the general level of the constructible representations of $\ca_s(P)$.

\subsection{On reduced C*-algebras}

In \cite{Li17} it has been established that $\ca_\la(P)$ coincides with the partial crossed product of a partial action of $G$ on the diagonal algebra $D_\la(P)$.
It coincides with the reduced C*-algebra of $\P_s$ if and only if $\J$ is independent.
We have three further reduced boundary quotients arising from: 

\vspace{1pt}

\begin{enumerate}[label=(\alph*)]
\item the partial crossed  product picture, i.e., $C(\partial \Om_P) \rtimes_r G$; \vspace{1pt}
\item the strong covariance Fell bundle, i.e., $\ca_\la(\P_{\scv})$; and \vspace{1pt}
\item the strong covariance relations in the left regular representation, i.e., $q_{\scv}(\ca_{\la}(P))$.
\end{enumerate}

\vspace{1pt}

\noindent
Items (a) and (b) are automatically $*$-isomorphic due to the realization of $\ca_{\scv}(P)$ as $C(\partial \Om_P) \rtimes_r G$.
Under exactness we can pass the results from the universal C*-algebras between items (a) and (c) down to the reduced case.
Indeed exactness induces a normal co-action on $q_{\scv}(\ca_{\la}(P))$ from $\ca_\la(P)$, and thus we can use the fixed-point-algebra property of $q_{\scv}(\ca_{\la}(P))$ inherited from $\ca_{\scv}(P)$ (see Theorem \ref{T:Fell bdl univ}).
This settles {\bf (Qn.\ 2)} and provides the upper-left part of Figure \ref{F:roadmap}.

\subsection{On co-universality}

We next turn our attention to {\bf (Qn.\ 3)}.
By \cite[Lemma 5.7.10]{CELY17} $\partial \Om_P$ embeds as the smallest, closed, $G$-invariant, nonempty subset of the spectrum of $D_s(P)$.
Consequently in Theorem \ref{T:co-universal} we deduce that the quotient map
\[
\ca_s(P) \to \partial \ca_\la(P)
\]
factors through any non-trivial equivariant representation of $\ca_s(P)$, i.e., $\partial \ca_\la(P)$ is co-universal for the non-trivial equivariant constructible representations of $P$.

On the other hand we connect $\partial \ca_\la(P)$ to a \v{S}ilov boundary.
The natural candidate for a nonselfadjoint algebra is the closed algebra $\A(P)$ generated by the image of $P$ inside $\ca_\la(P)$.
The normal coaction $\ol{\de}$ on $\ca_\la(P)$ descends to $\A(P)$ giving rise to the cosystem $(\A(P), G, \ol{\de})$, and in Theorem \ref{T:part} we establish that
\[
\partial \ca_\la(P) \simeq \cenv(\A(P), G, \ol{\de}).
\]

We further explore conditions under which $\cenv(\A(P), G, \ol{\de})$ is $\cenv(\A(P))$.
This is automatic for abelian semigroups, however it is unknown if in general $\cenv(\A(P))$ admits automatically a coaction of $G$.
By using the simplicity criteria of \cite{Li17} we find that this is the case when $P$ is an Ore semigroup or when $G$ acts topologically freely on $\partial \Om_P$.
In particular, if $P$ is an Ore semigroup then the C*-envelope is the usual $\ca_\la(G)$.
Surprisingly this is also a defining property for a semigroup $P \subseteq G$ to be Ore.
We thus derive the right-lower part of Figure \ref{F:roadmap}.

These diagrams expand on the theory of Ore semigroups in amenable groups.
In this case the scheme collapses into two distinct classes of $*$-isomorphic C*-algebras: (a) the Fock type ones
\[
\ca(\P_s) \simeq D_s(P) \rtimes_r G 
\simeq \ca_\la(P) \simeq D_\la(P) \rtimes_r G;
\]
and (b) their boundary quotients
\begin{align*}
\ca_{\scv}(P)  
& \simeq 
q_{\scv}(\ca_{\la}(P)) 
\simeq
\partial \ca_\la(P) := C(\partial \Om_P) \rtimes_r G
\simeq 
\cenv(\A(P), G, \ol{\de})
\simeq
\cenv(\A(P)) 
\simeq
\ca_\la(G),
\end{align*}
which have the co-universal property with respect to $\ca_s(P)$.

\subsection*{Structure of sections}

In Section \ref{S:pre} we provide the preliminaries on boundary quotients and the C*-envelope.
In Section \ref{S:sgrp alg} we gather the constructions that arise from $\ca_s(P)$ and connect with Fell bundles and partial crossed products.
In passing we also provide the identifications for the corresponding full C*-algebras.
In Section \ref{S:nc bdy} we study the boundary quotient $\partial \ca_\la(P)$ and we establish its co-universal properties.

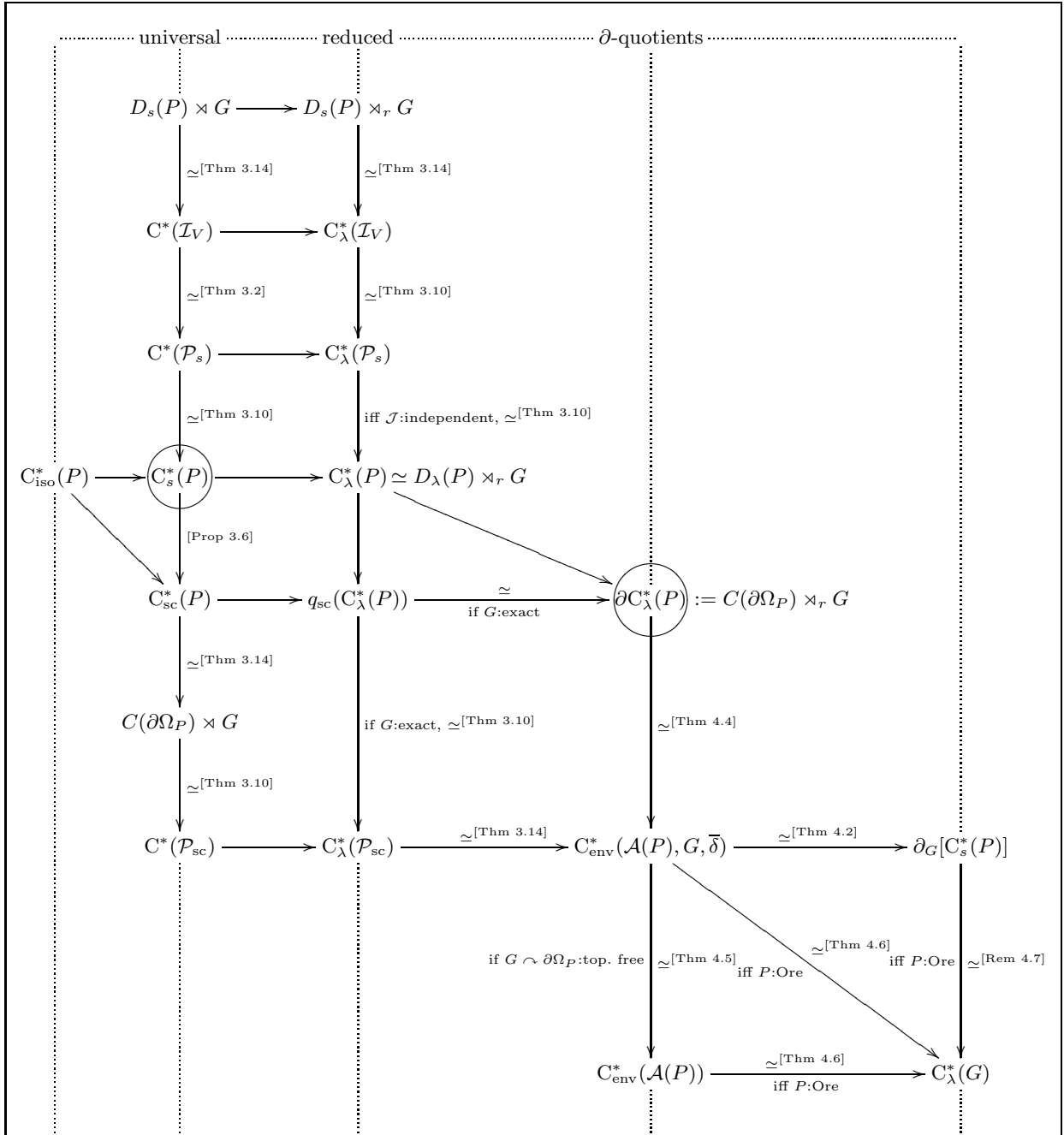
\begin{figure}[p!]
{\footnotesize
\[
\xymatrix@R=.6cm@C=.3cm{
\ar@{..}[r] \ar@{..}[ddddddd] & 
\textup{universal} \ar@{..}[d] \ar@{..}[rr] & & 
\textup{reduced} \ar@{..}[d] \ar@{..}[rrrrr] & & & & & 
\textup{$\partial$-quotients} \ar@{..}[ddddddddd] \ar@{..}[rrr] & & & 
\ar@{..}[ddddddddddddd] \\
& D_s(P) \rtimes G \ar[dd]^{\simeq^{\textup{[Thm \ref{T:un pcp}]}}} \ar[rr] & & 
D_s(P) \rtimes_r G \ar[dd]^{\simeq^{\textup{[Thm \ref{T:un pcp}]}}} & &
& & & & & & \\
& & & & & & & & & & & \\
& \ca(\I_V) \ar[dd]^{\simeq^{\textup{[Thm \ref{T:inv sgrp}]}}} \ar[rr] & & 
\ca_\la(\I_V) \ar[dd]^{\simeq^{\textup{[Thm \ref{T:Fell bdl univ}]}}} & &
& & & & & & \\
& & & & & & & & & & & \\
& 
\ca(\P_s) \ar[dd]^{\simeq^{\textup{[Thm \ref{T:Fell bdl univ}]}}} \ar[rr] & & 
\ca_\la(\P_s) \ar[dd]^{\textup{iff $\J$:independent, } \simeq^{\textup{[Thm \ref{T:Fell bdl univ}]}}} & &
& & & & & &\\
& & & & & & & & & & & \\
\ca_{\textup{iso}}(P) \ar[r] \ar[ddr] \ar@{..}[ddddddddddd] & 
*+[Fo]{\ca_s(P)} \ar[dd]^{{}^{\textup{[Prop \ref{P:scv ex}]}}} \ar[rr] & & 
\ca_\la(P) \ar[ddrrrrr] \ar[dd] & & & & & 
\hspace{-5.9cm} \simeq D_\la(P) \rtimes_r G & & & & \\
& & & & & & & & & & & \\
&
\ca_{\scv}(P) \ar[dd]^{\simeq^{\textup{[Thm \ref{T:un pcp}]}}} \ar[rr] 
& &
q_{\scv}(\ca_{\la}(P)) \ar[dddd]^{\textup{if $G$:exact, } \simeq^{\textup{[Thm \ref{T:Fell bdl univ}]}}}
\ar[rrrrr]^{\simeq}_{\textup{if $G$:exact}} 
& & & & & 
*+[Fo]{\partial \ca_\la(P)} \ar[dddd]^{\simeq^\textup{[Thm \ref{T:part}]}}
& 
\hspace{-1.cm} := C(\partial \Om_P) \rtimes_r G & & \\
& & & & & & & & & & & \\
&
C(\partial \Om_P) \rtimes G \ar[dd]^{\simeq^{\textup{[Thm \ref{T:Fell bdl univ}]}}} & & & & & & & & & & \\
& & & & & & & & & & & \\
& 
\ca(\P_{\scv}) \ar[rr] \ar@{..}[ddddd] & &
\ca_\la(\P_{\scv}) \ar[rrrrr]^{\simeq^\textup{[Thm \ref{T:un pcp}]}} \ar@{..}[ddddd]
& & & & &
\cenv(\A(P), G, \ol{\de})
\ar[dddd]^{\simeq^\textup{[Thm \ref{T:top free}]}}_{\textup{if $G \curvearrowright \partial \Om_P$:top.\ free}}
\ar[rrr]^{\phantom{ooo} \simeq^\textup{[Thm \ref{T:co-universal}]}} 
\ar[ddddrrr]^{\hspace{.2cm} \simeq^\textup{[Thm \ref{T:Ore}]}}_{\phantom{oooo} \textup{iff $P$:Ore} \hspace{.2cm}}
& & &
\partial_G [\ca_s(P)] \ar[dddd]^{\simeq^\textup{[Rem \ref{R:Ore}]}}_{\textup{iff $P$:Ore}}\\
& & & & & & & & & & \\
& & & & & & & & & & \\
& & & & & & & & & & \\
& & & & & & & &
\cenv(\A(P)) \ar@{..}[d] \ar[rrr]^{\simeq^\textup{[Thm \ref{T:Ore}]}}_{\textup{iff $P$:Ore}}  & & &
\ca_\la(G) \ar@{..}[d] \\
\ar@{..}[rrrrrrrrrrr] & & & & & & & & & & &
}
\]
}
\caption{Diagram with main results.}
\label{F:roadmap}



{\footnotesize
\begin{flushleft}
$\bullet$ $\ca_\la(P)$: The reduced C*-algebra of $P$.

\vspace{1pt}

$\bullet$ $\A(P)$: The reduced semigroup algebra of $P$ in $\ca_\la(P)$.

\vspace{1pt}

$\bullet$ $\ca_{\textup{iso}}(P)$: the universal C*-algebra with respect to isometric representations of $P$. 

\vspace{1pt}

$\bullet$ $\ca_s(P)$: the universal semigroup algebra of $(P, \J)$ in the sense of \cite{Li12}.

\vspace{1pt}

$\bullet$ $\ca_{\scv}(P)$: the strong covariant algebra of the trivial product system on $P$ in the sense of \cite{Seh18}.

\vspace{1pt}

$\bullet$ $q_{\scv}(\ca_{\la}(P))$: The quotient of the reduced C*-algebra by strong covariant relations.

\vspace{1pt}

$\bullet$ $\rtimes G$ is the universal partial crossed product; $\rtimes_r G$ is the reduced partial crossed product.

\vspace{1pt}

$\bullet$ $\partial_G [\ca_s(P)]$: the co-universal C*-algebra with respect to $G$-equivariant representations of $\ca_s(P)$.

\vspace{1pt}

$\bullet$ $\cenv(\A(P), G, \ol{\de})$ is the C*-envelope of the cosystem $(\A(P), G, \ol{\de})$; $\cenv(\A(P))$ is C*-envelope of $\A(P)$.

\vspace{1pt}

$\bullet$ $\I_V$: the inverse semigroup induced by the left regular representation of $P$.

\vspace{1pt}

$\bullet$ Fell bundles: $\P_s$ induced in $\ca_s(P)$; $\P_{\scv}$ induced in $\ca_{\scv}(P)$.

\vspace{1pt}

$\bullet$ Spaces: $D_s(P) := [\ca_s(P)]_e$; $D_\la(P) := [\ca_\la(P)]_e$; $\Om_P = (D_\la(P))^*$; $\partial \Om_P$ is the minimal $G$-subspace in $\Om_P$.

\vspace{1pt}
\end{flushleft}
}

\end{figure}

\begin{acknow}
Part of the research was carried out at the Banff International Research Station during the Focused Research Group week on ``Noncommutative boundaries for tensor algebras'' (20frg248).
Evgenios Kakariadis acknowledges support from EPSRC as part of the programme ``Operator Algebras for Product Systems'' (grant No.\ EP/T02576X/1).
Marcelo Laca was partially supported by NSERC Discovery Grant RGPIN-2017-04052.
Xin Li has received funding from the European Research Council (ERC) under the European Union's Horizon 2020 research and innovation programme (grant agreement No.\ 817597).
We would like to thank Adam Dor-On for several helpful discussions during the initial stages of the project.
We also thank Camila F.\ Sehnem for pointing out that the application of our Theorem \ref{T:top free} to $ax+b$-semigroups of rings was more general than originally stated, leading to the updated Remark~\ref{rem:finalremark}(vi).
\end{acknow}

\section{Preliminaries} \label{S:pre} 

\subsection{The \v{S}ilov boundary}

The reader may refer to \cite{Pau02} for the general theory of nonselfadjoint operator algebras and dilations of their representations, which we will avoid repeating here in full length.

Let $\A$ be an operator algebra, which in this paper means a closed subalgebra of $\B(H)$ for a Hilbert space $H$.
We say that $(C, \iota)$ is a \emph{C*-cover} of $\A$ if $\iota \colon \A \to C$ is a completely isometric representation with $C = \ca(\iota(\A))$.
The \emph{C*-envelope} $\cenv(\A)$ of $\A$ is a C*-cover $(\cenv(\A), \iota)$ with the following co-universal property:
if $(C', \iota')$ is a C*-cover of $\A$ then there exists a (necessarily unique) $*$-epimorphism $\Phi \colon C' \to \cenv(\A)$ such that $\Phi(\iota'(a)) = \iota(a)$ for all $a \in \A$.
Arveson defined the C*-envelope in \cite{Arv69} and computed it for a variety of operator algebras, predicting its existence in general.
Ten years later Hamana \cite{Ham79} confirmed Arveson's prediction by proving the existence of injective envelopes for the unital case.
The C*-envelope is the C*-algebra generated in the injective envelope of $\A$ once this is endowed with the Choi-Effros C*-structure. 

Dritschel and McCullough \cite{DM05} provided an alternative proof based on maximal dilations for the unital case.
A \emph{dilation} of a representation $\phi \colon \A \to \B(H)$ is a representation $\phi' \colon \A \to \B(H')$ such that $H \subseteq H'$ and $\phi(a) = P_H \phi'(a) |_H$ for all $a \in \A$.
A completely contractive map $\phi \colon \A \to \B(H)$ is called \emph{maximal} if every dilation $\phi' \colon \A \to \B(H')$ is trivial, i.e., $P_H \phi'(a) =\phi(a) = \phi'(a) |_H$ for all $a \in \A$.
It follows that the C*-envelope is the C*-algebra generated by a maximal completely isometric representation.
It does not hold in general that if $\pi \colon \cenv(\A) \to \B(H)$ is a $*$-representation then it is the unique contractive completely positive (ccp) extension of $\pi|_{\A}$.
The algebra $\A$ is called \emph{hyperrigid} if this is the case for any representation $\pi$ of $\cenv(\A)$.

The basic examples of C*-envelopes arise in the context of uniform algebras: the C*-envelope of a uniform algebra is formed by the continuous functions on its \v{S}ilov boundary.
The unconditional existence of the C*-envelope provides a non-commutative analogue of this result.
Consider $\A \subseteq \ca(\A)$.
An ideal $\I \lhd \ca(\A)$ is called a \emph{boundary ideal} if the quotient map $q_{\I} \colon \ca(\A) \to \ca(\A)/\I$ restricts to a completely isometric map on $\A$.
The \emph{\v{S}ilov ideal} $\I_s$ is by definition the boundary ideal that contains all boundary ideals of $\A$.
The existence of the C*-envelope implies the existence of the \v{S}ilov ideal; in particular it follows that $\cenv(\A)$ is canonically isomorphic to $\ca(\A)/\I_s$.

\subsection{Coactions on operator algebras}

We denote the minimal tensor product by $\otimes$.
We will need some elements about coactions on C*algebras as well as some results from \cite{DKKLL20} about coactions on operator algebras.

For a discrete group $G$ we write $u_g$ for the unitary generator associated with $g \in G$ in the full group C*-algebra $\ca(G)$.
We write $\la_g$ for the generators of the left regular representation $\ca_\la(G)$.
Recall that $\ca(G)$ admits a faithful $*$-homomorphism
\[
\De \colon \ca(G) \to \ca(G) \otimes \ca(G) ; u_g \mapsto u_g \otimes u_g.
\]
On the other hand $\ca_\la(G)$ admits a faithful $*$-homomorphism
\[
\De_\la \colon \ca_\la(G) \to \ca_\la(G) \otimes \ca_\la(G) ; \la_g \mapsto \la_g \otimes \la_g.
\]

\begin{definition}\label{D:cis ncoa} \cite[Definition 3.1]{DKKLL20}
Let $\A$ be an operator algebra.
A \emph{coaction of $G$ on $\A$} is a completely isometric representation $\de \colon \A \to \A \otimes \ca(G)$ such that the linear span of the induced subspaces
\[
\A_g := \{a \in \A \mid \de(a) = a \otimes u_g\}
\] 
is norm-dense in $\A$, in which case $\de$ satisfies the coaction identity
\[
(\de \otimes \id_{\ca(G)}) \de = (\id_{\A} \otimes \De) \de.
\]
If, in addition, the map $(\id \otimes \la) \de$ is injective then the coaction $\de$ is called \emph{normal}. 

If $\A$ is an operator algebra and $\de \colon \A \to \A \otimes \ca(G)$ is a coaction on $\A$, then we will refer to the triple $(\A, G, \de)$ as a \emph{cosystem}. 
A map $\phi \colon \A \to \A'$ between two cosystems $(\A, G, \de)$ and $(\A', G, \de')$ is said to be \emph{$G$-equivariant}, or simply equivariant, if $\de' \phi=(\phi\otimes \id)\de$.
\end{definition}

If $(\A, G, \de)$ is a cosystem then $\A_r \cdot \A_s \subseteq \A_{r s}$ for all $r, s \in G$, since $\de$ is a homomorphism.

\begin{remark} \cite{DKKLL20}
A coaction $\de$ of $G$ on $\A$ is automatically \emph{non-degenerate}, in the sense that 
\[
\ol{\de(\A) \left[I \otimes \ca(G)\right]} = \A \otimes \ca(G).
\]
In particular suppose that $\de \colon \ca(\A) \to \ca(\A) \otimes \ca(G)$ is a $*$-homomorphism satisfying the coaction identity
\[
(\de \otimes \id) \de(c) = (\id \otimes \De) \de(c) \foral c \in \ca(\A),
\]
and $(\A, G, \de|_{\A})$ is a cosystem.
Then $\de$ is automatically non-degenerate on $\ca(\A)$, i.e.,
\[
\ol{\de(\ca(\A)) \left[\ca(\A) \otimes \ca(G)\right]} = \ca(\A) \otimes \ca(G).
\]
In particular the definition of the coaction here extends that of a full coaction on a C*-algebra by Quigg \cite{Qui96}.
\end{remark}

\begin{remark}\label{R:reduced} \cite{DKKLL20}
Suppose that $\A$ admits a ``reduced'' coaction in the sense that there is a faithful map $\de_\la \colon \A \to \A \otimes \ca_\la(G)$ that satisfies the coaction identity
\[
(\de_\la \otimes \id_{\ca_\la(G)}) \de_\la(a) = (\id_{\ca(\A)} \otimes \De_\la) \de_\la(a) \foral a \in \A,
\]
and for which the linear span of the induced subspaces $\A_g := \{a \in \A \mid \de_\la(a) = a \otimes \la_g\}$ is norm-dense in $\A$.
Due to Fell's absorption principle, $\de_\la$ promotes to a normal coaction $\de$ of $G$ on $\A$ such that $\de_\la = (\id \otimes \la) \de$.
\end{remark}

\begin{definition}\label{D:co-action} \cite[Definition 3.6]{DKKLL20}
Let $(\A, G, \de)$ be a cosystem.
A triple $(C, \iota, \de_C)$ is called a \emph{C*-cover} for $(\A, G, \de)$ if $(C, \iota)$ is a C*-cover of $\A$ and $\de_C \colon C \to C \otimes \ca(G)$ is a coaction on $C$ such that the diagram
\[
\xymatrix{
\A \ar[rr]^{\iota} \ar[d]^{\de} & & C \ar[d]^{\de_C} \\
\A \otimes \ca(G) \ar[rr]^{\iota \otimes \id} & & C \otimes \ca(G)
}
\]
commutes.
\end{definition}

\begin{definition} \cite[Definition 3.7]{DKKLL20}
Let $(\A, G, \de)$ be a cosystem.
The \emph{C*-envelope} of $(\A, G, \de)$ is a C*-cover $(\cenv(\A, G, \de), \iota, \de_{\env})$ such that: for every C*-cover $(C', \iota', \de')$ of $(\A, G, \de)$ there exists a $*$-epimorphism $\Phi \colon C' \to \cenv(\A, G, \de)$ that fixes $\A$ and intertwines the coactions, i.e., the diagram
\[
\xymatrix{
\iota'(\A) \ar[rrr]^{\de'} \ar[d]^{\Phi} & & & C' \otimes \ca(G) \ar[d]^{\Phi \otimes \id} \\
\iota(\A) \ar[rrr]^{\de_{\env}} & & & \cenv(\A, G, \de) \otimes \ca(G)
}
\]
is commutative on $\A$, and thus is commutative on $C'$.
\end{definition}

The existence of the C*-envelope of a cosystem was proved in \cite{DKKLL20} by a direct computation  that uses the C*-envelope of the ambient operator algebra.
In order to state the result explicitly we need to make some preliminary remarks and establish the notation.
Suppose $(\A, G, \de)$ is a cosystem, let $i \colon \A \to \cenv(\A)$ be the C*-envelope of $\A$, and recall that the spatial tensor product of completely isometric maps is completely isometric.
Then the representation of $\A$  obtained via the  composition
\[
\xymatrix@C=2cm{
\A \ar[r]^{\de} & \A \otimes \ca(G) \ar[r]^{i \otimes \id} \ar[r] & \cenv(\A) \otimes \ca(G)
}
\]
is completely isometric, and the C*-algebra 
\[
\ca((i \otimes \id) \de(\A)) := \ca(i(a_g) \otimes u_g \mid g \in G)
\]
becomes a C*-cover of $\A$.
This C*-cover is special because it admits a coaction $\id \otimes \De$, so that the triple
\[
(\ca(i(a_g) \otimes u_g \mid g \in G), (i \otimes \id) \de, \id \otimes \De)
\]
becomes a C*-cover for $(\A, G, \de)$. 
The following theorem summarizes fundamental results about existence and representations of C*-envelopes for cosystems.

\begin{theorem} \label{T:DKKLL} \cite[Theorem 3.8 and Corollary 3.10]{DKKLL20}
Let $(\A, G, \de)$ be a cosystem and let $i \colon \A \to \cenv(\A)$ be the inclusion map.
Then
\[
(\cenv(\A, G, \de), \iota, \de_{\env}) \simeq (\ca(i(a_g) \otimes u_g \mid g \in G), (i \otimes \id )\de, \id \otimes \De).
\]
If in addition $\de$ is normal on $\A$ then $\de_{\env}$ is normal on $\cenv(\A, G, \de)$.

Moreover if $\Phi \colon \cenv(\A, G, \de) \to B$ is a $*$-homomorphism that is completely isometric on $\A$ then it is faithful on the fixed point algebra of $\cenv(\A, G, \de)$.
\end{theorem}

Let us close this section with some remarks on topological gradings from \cite{Exe97, Exe17}.
Recall that a \emph{topological grading} $\{\B_g\}_{g \in G}$ of a C*-algebra $\B$ consists of linearly independent subspaces that span a dense subspace of $\B$ and are compatible with the group $G$, i.e., $\B_g^* = \B_{g^{-1}}$ and $\B_g \cdot \B_h \subseteq \B_{g h}$.
By \cite[Theorem 3.3]{Exe97} the linear independence condition can be substituted by the existence of a conditional expectation on $\B_e$.
The maximal C*-algebra $\ca(\B)$ of $\B$ is defined as universal with respect to the representations of $\B$.
The reduced C*-algebra $\ca_\la(\B)$ of $\B$ is defined by the left regular representation of $\B$ on $\ell^2(\B)$.
If $\Phi$ is a representation of $\B$ then the range of $\Phi$ has a natural topological grading $\Phi(\B) = \{\Phi(\B_g)\}_{g \in G}$.
By \cite[Proposition 21.3]{Exe17} there are equivariant $*$-homomorphisms making the following diagram
\[
\xymatrix{
\ca_{\max}(\B) \ar[rr] \ar[d] & & \ca_\la(\B) \ar[d] \\
\ca_{\max}(\Phi(\B)) \ar[r] & \ca(\Phi) \ar[r] & \ca_\la(\Phi(\B))
}
\]
commutative.
A topological grading defines a \emph{Fell bundle} and once a representation of a Fell bundle is established the two notions are the same.
In a loose sense a Fell bundle $\B$ over a discrete group $G$ is a collection of Banach spaces $\{\B_g\}_{g \in G}$, often called \emph{the fibers of $\B$}, that obey to canonical algebraic properties and the C*-norm properties; see \cite[Definition 16.1]{Exe17}.
So we will alternate between these two notions.
Spectral subspaces of coactions on C*-algebras are an important source of topological gradings.

\begin{definition}
Let $\B = \{C_g\}_{g \in G}$ be a topological grading for a C*-algebra $C$ over a group $G$. 
We say that an ideal $\I \lhd C$ is \emph{induced} if $\I = \sca{\I \cap C_e}$.
\end{definition}

If $\de \colon C \to C \otimes \ca(G)$ is a coaction on a C*-algebra and $\I \lhd C$ is an induced ideal then $\de$ induces a faithful coaction $C / \I$, see for example \cite[Proposition A.1]{CLSV11}. 
Normal actions also descend through induced ideals when $G$ is exact, see for example \cite[Proposition A.5]{CLSV11}. 

\begin{definition}
Let $\de$ be a coaction of $G$ on a C*-algebra $C$ and let $\I \lhd C$ be an ideal of $C$.
We say that the quotient map is \emph{$G$-equivariant}, or that the quotient $C/\I$ is \emph{$G$-equivariant} if $\de$ descends to a coaction of $G$ on $C/\I$.
\end{definition}

\section{Semigroup algebras} \label{S:sgrp alg} \addtocontents{toc}{\protect\setcounter{tocdepth}{2}}

In this section we present the concrete and the universal C*-algebras that are related to semigroups, partial crossed products and product systems.
We provide the identification of several universal and reduced C*-algebras that arise in this context.

\subsection{The reduced semigroup algebra}

Let $P$ be a unital semigroup in a discrete group $G$.
We write
\[
\ca_\la(P) := \ca(V_p \mid p \in P)
\qand
\A(P) := \ol{\alg}\{V_p \mid p \in P\}
\]
for the C*-algebra and the operator algebra, respectively, generated by the left-creation operators
\[
V_p \colon \ell^2(P) \to \ell^2(P) ; \de_s \mapsto \de_{ps}.
\]
Let $U \colon \ell^2(P) \otimes \ell^2(G) \to \ell^2(P) \otimes \ell^2(G)$ be the unitary operator determined by
\[
U(\de_s \otimes \de_g) = \de_s \otimes \de_{sg} \foral s \in P, g \in G,
\]
and let $\la_g$ be the canonical unitary corresponding to $g \in G$ in the left regular representation of $G$.
A routine calculation shows that
\[
( V_p \otimes \la_p ) U = U (V_p \otimes I)
\foral 
p \in P.
\]
Thus the $*$-homomorphism $\ol{\de}_\la$ obtained by composition
\[
\xymatrix{
\ca_\la(P) \ar[rr]^{\simeq \phantom{ooooooo}} & &
\ca(V_p \otimes I \mid p \in P) \ar[rr]^{\ad_{U}} & &
\ca(V_p \otimes \la_p \mid p \in P)
}
\]
is faithful and satisfies the coaction identity.
We also note that
\begin{align*}
[\ca_\la(P)]_g 
& := 
\{a \in \ca_\la(P) \mid \ol{\de}_\la(a) = a \otimes \la_g\} \\
& \supseteq 
\spn\{V_{p_1}^* V_{q_1} \cdots V_{p_n}^* V_{q_n} \mid p_1^{-1} q_1 \cdots p_n^{-1} q_n = g; n \in \bZ_+; p_i, q_i \in P\},
\end{align*}
and thus by construction these fibers are norm-dense in $\ca_\la(G)$.
The reverse inclusion, and hence equality, follows by applying $\id \otimes E_g$, where $E_g$ is the $g$-th Fourier coefficient map on $\ca_\la(G)$.
In particular by restricting to $\A(P)$ we get that
\[
[\A(P)]_p = \bC \cdot V_p \foral p \in P
\qand
[\A(P)]_g = (0) \foral g \notin P.
\]

Let $\ol{\de} \colon \ca_\la(P) \to \ca_\la(P) \otimes \ca(G)$ be the normal coaction induced by $\ol{\de}_\la$.
The induced faithful conditional expectation $E_\la$ is implemented by compressing $(\id \otimes \la) \ol{\de}$ to the $(e,e)$-entry and thus
\[
E_\la(V_{p_1}^* V_{q_1} \cdots V_{p_n}^* V_{p_n})
=
\begin{cases}
V_{p_1}^* V_{q_1} \cdots V_{p_n}^* V_{q_n} & \text{if } p_1^{-1}q_1 \cdots p_n^{-1} q_n = e,\\
0 & \text{otherwise}.
\end{cases}
\]
In \cite{Li12} it has been established that 
\[
E_\la(V_{p_1}^* V_{q_1} \cdots V_{p_n}^* V_{q_n}) = \sum_{s \in P} Q_s (V_{p_1}^* V_{q_1} \cdots V_{p_n}^* V_{q_n})) Q_s,
\]
where $Q_s$ is the projection on $\bC \cdot \de_s$.
In other words, $E_\la$ is the faithful conditional expectation on $\ca_\la(P)$ given by compressing to the diagonal.
We will write $D_\la(P)$ for the fixed point algebra $[\ca_\la(P)]_e$ of $\ol{\de}$ on $\ca_\la(P)$.

We will require some additional facts from \cite{Li12}.
For a set $X \subseteq P$ and $p \in P$ we write
\begin{equation}\label{eqn:pX&pinvX}
pX:=\{px \mid x \in X\}
\qand
p^{-1}X := \{y \in P \mid py \in X\}.
\end{equation}
Note here that by definition $p^{-1} P = P$.
We write $\J$ for the smallest family of right ideals of $P$ containing $P$ and $\mt$ that is closed under left multiplication and taking pre-images under left multiplication (as in the sense above), i.e.,
\[
\J := \left\{ p_1^{-1} q_1 \dots p_n^{-1} q_n P \mid n \in \bZ_+; p_i, q_i \in P, 1 \leq i \leq n \right\} \cup \{\mt\}.
\]
The elements in $\J$ are called \emph{constructible} right ideals of $P$.
It is important to notice  that a constructible right ideal $p_{1}^{-1} q_{1} \dots p_{ n}^{-1} q_{n} P$ does not depend on the product $ p_{1}^{-1} q_{1} \dots p_{ n}^{-1} q_{n}$ as an element of $G$ because the second operation in \eqref{eqn:pX&pinvX} involves the  pre-image of multiplication in $P$ and not in $G$.
It follows from \cite[Lemma 3.3]{Li12} that 
\[
q_1^{-1} p_1 \dots q_m^{-1} p_m p_m^{-1} q_m \dots p_1^{-1} q_1 X = (q_1^{-1} p_1 \dots q_m^{-1} p_m P) \cap X
\]
for every finite collection $p_i, q_i \in P$ and every subset $X$ of $P$.
Thus the set of constructible ideals is actually automatically closed under finite intersections.
We will write $\Bx, \By, \Bz$ etc.\ for the elements of $\J$.
For a set $X \subseteq P$ we will write $E_{[X]}$ for the projection on the subspace $\ell^2(X)$ of $\ell^2(P)$.

\begin{proposition}\label{P:proj ideal} \cite[Lemma 3.1]{Li12}
Let $P$ be a unital semigroup in a group $G$.
For elements $p_1, q_1, \dots, p_n, q_n \in P$ we have that
\[
E_{[p_1^{-1} q_1\dots p_n^{-1} q_n P]} = V_{p_1}^* V_{q_1} \cdots V_{p_n}^* V_{q_n} V_{q_n}^* V_{p_n} \cdots V_{q_1}^* V_{p_1} \in \ca_\la(P).
\]
If $p_1^{-1} q_1 \cdots p_n^{-1} q_n = e_G$ in $G$ then
\[
V_{p_1}^* V_{q_1} \cdots V_{p_n}^* V_{q_n} = E_{[q_n^{-1} p_n \dots q_1^{-1} p_1 P]}.
\] 
Moreover, we have that
\[
E_{[\Bx]} E_{[\Bx']} = E_{[\Bx \cap \Bx']}
\foral
\Bx, \Bx' \in \J.
\]
Consequently, we have that
\[
D_\la(P) := [\ca_\la(P)]_e = \ol{\spn}\{ E_{[\Bx]} \mid \Bx \in \J\}
\]
 for the fixed point algebra of $\ca_\la(P)$.
\end{proposition}

Notice that if a finite set $F \subseteq \J$ of constructible right ideals  is closed under intersection then 
\[
B_F := \spn\{E_{[\Bx]} \mid \Bx \in F\} 
\]
is a finite-dimensional, hence closed, $*$-subalgebra of the diagonal $D_\la(P) $. By saturating every finite subset under intersection, we see that $D_\la(P) $ is the inductive limit of the $ B_F $  over the set of  $\cap$-closed, finite subsets of $\J$ directed by inclusion.
Following \cite{Li12} we say that $\J$ is \emph{independent} when the following holds: 
\begin{center}
\textit{for all $\Bx, \Bx_1, \dots, \Bx_n \in \J$ with $\bigcup_{i=1}^n \Bx_i = \Bx$ there is an $i_0$ such that $\Bx = \Bx_{i_0}$.}
\end{center}
It then follows that $\J$ is independent if and only if $\{E_{[\Bx]}\}_{\Bx \in F}$ is a basis for $B_F$ \cite[Corollary 5.6.29]{CELY17}.

\subsection{Universal semigroup C*-algebras}

The \emph{isometric semigroup C*-algebra} $\ca_{\textup{iso}}(P)$ is the universal C*-algebra generated by isometries $\{\fv_p \mid p \in P\}$ satisfying $\fv_p \fv_q = \fv_{pq}$.
Universality implies that there exists a coaction of $G$ on $\ca_{\textup{iso}}(P)$ determined by $\fv_p \mapsto \fv_p \otimes u_p$ and having fixed point algebra
\[
[\ca_{\textup{iso}}(P)]_e = \ol{\spn} \{\fv_{p_1}^* \fv_{q_1} \cdots \fv_{p_n}^* \fv_{q_n} \mid p_1^{-1} q_1 \cdots p_n^{-1} q_n = e \}.
\]

The \emph{full semigroup C*-algebra} $\ca_{\textup{full}}(P)$ of $P$ is the universal C*-algebra generated by isometries $\{\mathcal{V}_p \mid p \in P\}$ and projections $\{\mathcal{E}_{\Bx} \mid \Bx \in \J\}$ satisfying the relations:
\begin{enumerate}
\item[I.] $\mathcal{V}_{p q} = \mathcal{V}_p \mathcal{V}_q$ and $\mathcal{V}_p \mathcal{E}_{\Bx} \mathcal{V}_p^* = \mathcal{E}_{p \Bx}$;
\item[II.] $\mathcal{E}_P = I$,  $\mathcal{E}_\mt = 0$; and $\mathcal{E}_{\Bx} \cdot \mathcal{E}_{\By} = \mathcal{E}_{\Bx \cap \By}$.
\end{enumerate}
Introduced in \cite{Li12}, $\ca_{\textup{full}}(P)$ offers a model of $\ca_\la(P)$ where the projections $\E_{\Bx}$ corresponding to the constructible ideals are regenerated in the obvious way, i.e.,
\[
\E_{\Bx} = \V_{p_1}^* \V_{q_1} \cdots \V_{p_n}^* \\V_{q_n} V_{q_n}^* \V_{p_n} \cdots \V_{q_1}^* \V_{p_1}
\qfor
\Bx = p_1^{-1} q_1 \dots p_n^{-1} q_n P.
\]
The full semigroup C*-algebra admits a coaction of $G$ whose fixed point algebra \emph{contains} the commutative C*-algebra
\[
D_{\textup{full}}(P) := \ca(\E_{\Bx} \mid \Bx \in \J).
\]
In \cite{Li12} it is shown that the canonical $*$-epimorphism $\ca_{\textup{full}}(P) \to \ca_\la(P)$ is faithful on $D_{\textup{full}}(P)$ if and only if $\J$ is independent.

The \emph{constructible semigroup C*-algebra} $\ca_s(P)$ of $P$ introduced in \cite[Definition 3.2]{Li12} is the universal C*-algebra generated by isometries $\{v_p \mid p \in P\}$ and projections $\{e_{\Bx} \mid \Bx \in \J\}$ satisfying the relations:
\begin{enumerate}
\item[I.] $v_{p q} = v_p v_q$;
\item[II.] $e_\mt = 0$;
\item[III${}_G$.] if $p_1, q_1, \dots, p_n, q_n$ satisfy $p_1^{-1} q_1 \cdots p_n^{-1} q_n = e_G$ then 
\[
v_{p_1}^* v_{q_1} \cdots v_{p_n}^* v_{q_n} = e_{\Bx} \qfor \Bx = q_n^{-1} p_n \dots q_1^{-1} p_1 P.
\]
\end{enumerate}
It follows by \cite[Lemma 3.3]{Li12} that the family $\{v_p, e_{\Bx} \mid p \in P, \Bx \in \J\}$ satisfies also the relations
\[
e_P = 1, \, v_p e_{\Bx} v_p^* = e_{p \Bx}, \, \text{and} \, e_{\Bx} \cdot e_{\By} = e_{\Bx \cap \By}.
\]
Therefore $\ca_s(P)$ is a quotient of $\ca_{\textup{full}}(P)$.
Universality implies that there exists a coaction of $G$ on $\ca_s(P)$ that induces a conditional expectation
\[
E \colon \ca_s(P) \to \ol{\spn}\{ v_{p_1}^* v_{q_1} \cdots v_{p_n}^* v_{q_n} \mid p_1^{-1} q_1 \cdots p_n^{-1} q_n = e\}.
\]
Hence the projections $e_{\Bx}$ have dense linear span $[\ca_s(P)]_e$.
We will write
\[
D_s(P) : = \ca( e_{\Bx} \mid \Bx \in \J ) = [\ca_s(P)]_e.
\]
In particular the fixed point algebra $[\ca_s(P)]_e$ is the inductive limit of the (finite dimensional and thus closed) C*-subalgebras
\[
\B_F := \spn\{ e_{\Bx} \mid \Bx \in F\}
\textup{ for finite $\cap$-closed }
F \subseteq \J.
\]
If $\J$ is independent then the canonical $*$-epimorphism $\ca_s(P) \to \ca_\la(P)$ is faithful on $D_s(P)$, by \cite[Corollary 2.27]{Li12}.
We will prove the converse of that, by using a result of Norling \cite{Nor14}.

Towards this end we need to make a connection with inverse semigroups.
Recall  that if $\S$ is an inverse semigroup then the reduced C*-algebra  $\ca_\la(\S)$ is the C*-algebra generated by the operators
$
\la(s) \colon \ell^2(\S \setminus \{0\}) \to \ell^2(\S \setminus \{0\}), 
$
determined by  
\[
\la(s) \de_x =
\begin{cases}
 \de_{sx} & \text{ if } s^*s \geq xx^*,\\ 0 &\text{otherwise.}
\end{cases}
\]
There is also a universal C*-algebra $\ca(\S)$ generated by a universal representation $\{u_s: s\in \S\}$ of $\S$ by partial isometries.
The fastest way to obtain an inverse semigroup from our $P$ is to use its left regular representation $V$ and define
\[
\I_V := \{V_{p_1}^* V_{q_1} \cdots V_{p_n} V_{q_n}^* \mid n \in \bZ_+; p_i, q_i \in P, 1 \leq i \leq n\}.
\]
Then $\I_V$ is an inverse semigroup (of partial isometries on $\ell^2(P)$), so we have two C*-algebras
\[
\ca(\I_V) = \ca(u(V) \mid V \in \I_V)
\qand
\ca_\la(\I_V) = \ca(\la(V) \mid V \in \I_V).
\]
It was shown in \cite{Li13} that under the assumption that $P$ satisfies the Toeplitz condition from \cite[Definition 5.8.1]{Li17}, there is a canonical isomorphism $\ca_s(P) \cong \ca(\I_V)$. 
We remove this assumption next.

\begin{theorem}\label{T:inv sgrp}
Let $P$ be a unital subsemigroup in a group $G$ and let $\I_V$ be the induced inverse semigroup in $\ca_\la(P)$.
Then there exists a canonical $*$-isomorphism
\[
\ca_s(P) \to \ca(\I_V) ; v_p \mapsto u(V_p)
\]
that restricts to an isomorphism of  $D_s(P)$ to the C*-subalgebra $\ca(E)$ of $\ca(\I_V)$ generated by the semilattice of idempotents in $\I_V$.
\end{theorem}

\begin{proof}
Since the universal representation $u$ of $\I_V$ is multiplicative and unital, the composition 
\[
p \in P \mapsto V_p \in \I_V \mapsto u(V_p) \in \ca(\I_V)
\]
is clearly an isometric representation of $P$ in $\ca(\I_V)$.
Recall now that if $p_1^{-1} q_1 \cdots p_n^{-1} q_n = e_G$ and $\Bx = q_n^{-1} p_n \dots q_1^{-1} p_1 P$, then the relation 
\[
E_{[\Bx]} = V_{p_1}^* V_{q_1} \cdots V_{p_n}^* V_{q_n},
\]
holds in $\ca_\la(P)$. 
Thus $E_{[\Bx]} \in \I_V$ and the composition map
\[
\Bx \in \J \mapsto E_{[\Bx]} \in \I_V \mapsto u(E_{[\Bx]}) \in \ca(\I_V)
\]
satisfies 
\[
u(E_{[\Bx]}) = u(V_{p_1}^* V_{q_1} \cdots V_{p_n}^* V_{q_n}) = u(V_{p_1})^* u(V_{q_1}) \cdots u(V_{p_n})^* u(V_{q_n}).
\]
Thus the families $\{u(V_p) \mid p \in P\}$ and $\{u(E_{[\Bx]}) \mid \Bx \in \J\}$ satisfy the relations defining $\ca_s(P)$, and by the universal property there is a $*$-homomorphism $\ca_s(P) \to \ca(\I_V)$ mapping $v_p$ to $u(V_p)$.

For the other direction we want to show that the assignment
\begin{equation}\label{eqn:welldefindedness}
V_{p_1}^*V_{q_1} \cdots V_{p_n}^* V_{q_n} \in \I_V \mapsto v_{p_1}^* v_{q_1} \cdots v_{p_n}^* v_{q_n} \in \ca_s(P)
\end{equation}
is a well-defined representation of $\I_V$ by partial isometries.
To this end let us set
\[
V := V_{p_1}^* V_{q_1} \cdots V_{p_n}^* V_{q_n}
\qand
W := V_{r_1}^* V_{s_1} \cdots V_{r_m}^* V_{s_m}
\]
and suppose that $V = W$.
Hence we also have that $V V^* = W W^* = W V^* = V W^*$, 
and so
\[
q_n^{-1} p_n \dots q_1^{-1} p_1 P = s_m^{-1} r_m \dots s_1^{-1} r_1 P =: \Bx.
\]
For $V \neq 0$ we have that $\Bx \neq \mt$, and in particular we get that
\[
V W^* \de_t = V V^* \de_t = E_{[\Bx]} \de_t = \de_t \foral t \in \Bx.
\]
For any such $t \in \Bx$ we deduce that
\[
q_n^{-1} p_n \cdots q_1^{-1} p_1 s_m^{-1} r_m \cdots s_1^{-1} r_1 t = t,
\]
which yields $q_n^{-1} p_n \cdots q_1^{-1} p_1 s_m^{-1} r_m \cdots s_1^{-1} r_1 = e_G$.
Consequently we obtain another description for $\Bx$ by using the left regular representation.
That is, the relation
\[
E_{[\Bx]}
= 
VV^*
=
V W^*
=
E_{[q_n^{-1} p_n \cdots q_1^{-1} p_1 s_m^{-1} r_m \cdots s_1^{-1} r_1 P]}
\]
yields
\[
\Bx = q_n^{-1} p_n \cdots q_1^{-1} p_1 s_m^{-1} r_m \cdots s_1^{-1} r_1 P.
\]
Likewise we obtain similar expressions for $\Bx$ corresponding to the equations $E_{[\Bx]} = WV^* = WW^*$.
Let us set
\[
v := v_{p_1}^* v_{q_1} \cdots v_{p_n}^* v_{q_n}
\qand
w := v_{r_1}^* v_{s_1} \cdots v_{r_m}^* v_{s_m}.
\]
Then by the properties of $\ca_s(P)$ and the descriptions for $\Bx$ obtained above we get
\[
e_{\Bx} = v v^* = w w^* = w v^* = v w^*.
\]
This shows that $v$ and $w$ are partial isometries.
By considering the dual equalities
\[
V^* V = W^* W = V^* W = W^* V,
\]
we have the symmetrical
\[
v^* v = w^* w = v^* w = w^* v.
\] 
Since this implies
\[
v = v (v^* v) = v (w^*w) = (v w^*) w = (w w^*) w = w,
\]
the map given in \eqref{eqn:welldefindedness} is well defined and determines a $*$-epimorphism $\ca(\I_V) \to \ca_s(P)$.
Choosing $p_1=e_G$ and $q_1=p$ shows that this epimorphism is indeed the inverse of the $*$-homomorphism $ \ca_s(P) \to \ca(\I_V) $ constructed above, and the proof of the isomorphism $\ca_s(P) \simeq \ca(\I_V)$ is complete.

Finally, notice that by Proposition \ref{P:proj ideal}, the semilattice $E$ of idempotents in $\I_V$ coincides with the set of characteristic functions of constructible right ideals viewed as operators on $\ell^2(P)$.
Therefore $\ca(E) \subseteq \ca(\I_V)$ is the image of $D_s(P) \subseteq \ca_s(P)$ in the isomorphism.
\end{proof}

\begin{corollary}\label{C:nor-li} {\em (cf. \cite[Theorem 3.22]{Nor14} and \cite[Corollary 2.27]{Li12})} 
Let $P$ be a unital subsemigroup in a group $G$ and consider the commuting diagram
\[
\xymatrix{
\ca_{\textup{full}}(P) \ar[d] \ar[drr]^{\la_{\textup{full}}} & & \\
\ca_s(P) \ar[rr]^{\la_s} & & \ca_\la(P)
}
\]
of canonical $*$-epimorphisms.
Then the following are equivalent:
\begin{enumerate}
\item $\J$ is independent.
\item $\la_{\textup{full}}|_{D_{\textup{full}}(P)} \colon D_{\textup{full}}(P) \to D_\la(P)$ is faithful.
\item $\la_s|_{D_s(P)} \colon D_s(P) \to D_\la(P)$ is faithful.
\end{enumerate}
\end{corollary}

\begin{proof}
The equivalence of items (i) and (ii) is precisely \cite[Corollary 2.27]{Li12}. 
Since the diagram restricts to a commuting diagram of diagonal algebras, it is clear that (ii) implies (iii).  
In order to prove that (iii) implies (i), assume now that (iii) holds and observe that since $D_s(P) \simeq \ca(E)$ by Theorem \ref{T:inv sgrp}, the canonical $*$-epimorphism $\ca(E) \to D_\la(P)$ is faithful. 
Then \cite[Proposition 3.5 and Corollary 3.6]{Nor14} imply that (i) holds, see also the proof of \cite[Theorem 3.22]{Nor14}.
\end{proof}

\subsection{Sehnem's strong covariance relations}

Working at the generality of product systems, Sehnem \cite{Seh18} has provided a quotient of the Toeplitz C*-algebra with the following properties: (a) it admits an injective copy of the coefficient algebra; and (b) every equivariant $*$-representation of the quotient that is injective on the coefficient algebra is automatically faithful on the fixed point algebra.
Let us review her construction for a unital semigroup $P$ in a group $G$.
This corresponds to the product system $X$ over $P$ such that every fiber $X_p$ is $\bC$.
For a finite set $F \subseteq G$ set
\[
K_F := \bigcap_{g \in F} gP.
\]
For $r \in P$ and $g \in F$ define
\[
I_{r^{-1} K_{\{r,g\}}} :=
\begin{cases}
(0) & \text{if } K_{\{r,g\}} \neq \mt \text{ and } r \notin K_{\{r,g\}},\\
\bC & \text{otherwise},
\end{cases}
\qand
I_{r^{-1} (r \vee F)} := \bigcap_{g \in F} I_{r^{-1} K_{\{r,g\}}}.
\]
Let the spaces
\[
X_F := \sca{\de_r \in \ell^2(P) \mid I_{r^{-1} (r \vee F)} \neq (0) }
\qand
X_F^+ := \bigoplus_{g \in G} X_{gF} \subseteq \ell^2(P) \otimes \ell^2(G).
\]
For every $p \in P$ define the operator
\[
V_{F,p} \colon X_F^+ \to X_F^+ ; X_F \ni \de_r \mapsto \de_{pr} \in X_{p F}.
\]
It is well-defined as $I_{r^{-1}(r \vee F)} = I_{(pr)^{-1}(pr \vee pF)}$ for all $r \in P$, and $I_{r^{-1}(r \vee F)} = I_{(s^{-1}r)^{-1}(s^{-1}r \vee s^{-1}F)}$ for all $r \in sP$.
It follows that $V_{F,p}$ is an isometry and thus we obtain a $*$-representation
\[
\Phi_F \colon \ca_{\textup{iso}}(P) \to \B(X_F^+) ; \fv_p \mapsto V_{F,p}.
\]
For the projections 
\[
Q_{g,F} \colon X_F^+ \to X_{gF}
\]
we get that
\[
V_{F,p} Q_{g, F} = Q_{pg, F} V_{F,p}
\qand
V_{F,p}^* Q_{g,F} = Q_{p^{-1}g, F} V_{F,p}^*.
\]
Therefore $Q_{e, F}$ is reducing for $\Phi_F([\ca_{\textup{iso}}(P)]_e)$.
For $f \in [\ca_{\textup{iso}}(P)]_e$ define
\[
\nor{f}_F := \nor{Q_{e,F} \Phi_F(f) Q_{e,F}} = \nor{\Phi_F(f) Q_{e,F}} = \nor{Q_{e,F} \Phi_F(f)}.
\]
When we specialize the definition of strongly covariant representations from \cite{Seh18} to the case of 
semigroup algebras, which are obtained from product systems with one-dimensional fibers, we get the following.

\begin{definition} (cf. \cite[Definition 3.2]{Seh18})
Let $P$ be a unital subsemigroup of a group $G$.
An isometric representation of $P$ is called \emph{strongly covariant} if it vanishes on the ideal $\I_e \lhd [\ca_{\textup{iso}}(P)]_e$ given by
\[
\I_e := \{f \in [\ca_{\textup{iso}}(P)]_e \mid \lim_F \nor{f}_F = 0\},
\]
where the limit is taken with respect to the partial order induced by inclusion on finite sets $F$ of $G$.
We write $\I_\infty$ for the ideal of \emph{strong covariance relations} generated by $\I_e$.
We denote by $\ca_{\scv}(P)$ the universal C*-algebra with respect to the strong covariant representations of $X$.
\end{definition}

Since $\ca_{\scv}(P)$ is a quotient of $\ca_{\textup{iso}}(P)$ by an induced ideal, it inherits a faithful coaction $\ca_{\scv}(P) \to \ca_{\scv}(P) \otimes \ca(G)$.
Moreover since $\ca_{\scv}(P)$ is a unital $*$-representation of $\ca_{\textup{iso}}(P)$, we get that $\ca_{\scv}(P) \neq \{0\}$.
An important point of Sehnem's theory is that the image $v_p \in \ca_{\scv}(P)$ of each $p \in P$ is an isometry, and thus non-zero.
The next observation follows from \cite[Theorem 3.10]{Seh18}.

\begin{theorem}\label{T:inj ex}
Let $P$ be a unital semigroup in a group $G$.
Then a unital $*$-representation of $\ca_{\scv}(P)$ that admits a coaction of $G$ is automatically faithful on $[\ca_{\scv}(P)]_e$.
\end{theorem}

\begin{proof}
The conclusion follows by the property (C3) of $\ca_{\scv}(P)$ by \cite[Theorem 3.10]{Seh18}.
We note here that the coefficient algebra of the product system in this setting is $\bC$, and that trivially it embedded in $\ca_{\scv}(P)$.
\end{proof}

We see that every $\Phi_F$ is actually a restriction (and thus a sub-representation) of the ``reduced'' coaction of $\ca_\la(G)$ on $\ca_\la(P)$.
Since $Q_{e, F}$ is reducing for $\Phi_F([\ca_{\textup{iso}}(P)]_e)$ we obtain the maps
\begin{equation}\label{Eq:dia}
\xymatrix{
[\ca_{\textup{iso}}(P)]_e \ar[r]^{\rho} &
[\ca_s(P)]_e \ar[r]^{\la} &
[\ca_\la(P)]_e \ar[r]^{\Phi:= \bigoplus\limits_{\textup{fin } F \subseteq G} \Phi_F|_{X_F}}
&
\prod\limits_{\text{fin } F \subseteq G} \B(X_F) \ar[d]^{q_{\qd}} \\
& & &
\quo{\prod\limits_{\text{fin } F \subseteq G} \B(X_F)}{c_0(\B(X_F) \mid {\text{fin } F \subseteq G})}.
}
\end{equation}
Thus by definition it follows that $\I_\infty = \sca{\ker(q_{\qd} \Phi \la \rho) \cap [\ca_{\textup{iso}}(P)]_e}$.
Therefore every $e$-graded relation in $\ca_s(P)$ passes to the strong covariance algebra.
In particular we have the following proposition.

\begin{proposition} \label{P:scv ex}
Let $P$ be a unital subsemigroup in a group $G$.
Then the canonical $*$-epimorphism $\ca_{\textup{iso}}(P) \to \ca_{\scv}(P)$ factors through $\ca_s(P)$.
\end{proposition}

\begin{proof}
For convenience let us set $\ca_{\scv}(P) := \ca(\dot{\fv}_p \mid p \in P)$.
It is clear that the family $\{\dot{\fv}_p \mid p \in P\}$ is an isometric representation of $P$.
We have to show that there is a family of projections
\[
\{\dot{\fe}_{\Bx} \mid \Bx \in \J\} \subseteq \ca_{\scv}(P)
\]
which, along with the family of isometries $\{\dot{\fv}_p \mid p \in P\}$, satisfies the $e$-graded axioms of $\ca_s(P)$.
Recall that the projections in $\ca_s(P)$ can be recreated from the isometries, i.e., 
\[
e_{\Bx} = v_{p_1}^* v_{q_1} \cdots v_{p_n}^* v_ {q_n} v_{q_n}^* v_{p_n} \cdots v_{q_1}^* v_{p_1}
\qfor
\Bx = p_1^{-1} q_1 \dots p_n^{-1} q_n P.
\]
We are going to use this as a model.
Fix $p_1, q_1, \dots, p_m, q_m \in P$ and set 
\[
\ff := \fv_{p_1}^* \fv_{q_1} \cdots \fv_{p_n}^* \fv_ {q_n} \fv_{q_n}^* \fv_{p_n} \cdots \fv_{q_1}^* \fv_{p_1}.
\]
Clearly $\ff$ is selfadjoint.
Using the diagram (\ref{Eq:dia}) preceding the statement we see that
\[
\la \rho(\ff) = E_{[\Bx]} \qfor \Bx = p_1^{-1} q_1 \dots p_n^{-1} q_n P.
\]
As projections in C*-algebras are defined by $*$-algebraic relations, by strong covariance we then derive that the element
\[
\dot{\ff} := \dot{\fv}_{p_1}^* \dot{\fv}_{q_1} \cdots \dot{\fv}_{p_n}^* \dot{\fv}_ {q_n} \dot{\fv}_{q_n}^* \dot{\fv}_{p_n} \cdots \dot{\fv}_{q_1}^* \dot{\fv}_{p_1}
\]
is a projection in $\ca_{\scv}(P)$.
Indeed we have that $\ff - \ff^2 \in \I_e$ since
\[
q_{\qd} \Phi \la \rho(\ff - \ff^2)
=
q_{\qd} \Phi(E_{[\Bx]}) - q_{\qd} \Phi(E_{[\Bx]}^2)
=
0,
\]
and so $\dot{\ff} = \dot{\ff}^2$.

Next let $r_1, s_1, \dots, r_n, s_n$ such that
\[
\Bx := p_1^{-1} q_1 \dots p_m^{-1} q_m P = r_1^{-1} s_1 \dots r_n^{-1} s_n P
\]
and set 
\[
\fg := \fv_{r_1}^* \fv_{s_1} \cdots \fv_{r_n}^* \fv_ {s_n} \fv_{s_n}^* \fv_{r_n} \cdots \fv_{s_1}^* \fv_{r_1}.
\]
As before we see that $\la \rho(\fg) = E_{[\Bx]} = \la \rho(\ff)$ and so $\dot{\fg} = \dot{\ff}$.
Therefore there is a well-defined map
\[
\J \ni \Bx = p_1^{-1} q_1 \dots p_m^{-1} q_m P \mapsto \dot{\fe}_{\Bx} := \fv_{p_1}^* \fv_{q_1} \cdots \fv_{p_n}^* \fv_ {q_n} \fv_{q_n}^* \fv_{p_n} \cdots \fv_{q_1}^* \fv_{p_1}.
\]
Then the family $\{\dot{\fv}_p, \dot{\fe}_{\Bx} \mid p \in P, \Bx \in \J\}$ satisfies the axioms (I) and (II) of $\ca_s(P)$.
The third axiom III${}_G$ is $e$-graded and it is satisfied in $\ca_\la(P)$, thus it follows (with a similar argument to the one above) that this family satisfies axiom III${}_G$ as well.
Hence the proof is completed by the universal property of $\ca_s(P)$.
\end{proof}

Next we wish to introduce a reduced version of the strong covariance algebra. 

\begin{definition}
Let $P$ be a unital subsemigroup in a group $G$ and let $\la \rho \colon \ca_{\textup{iso}}(P) \to \ca_\la(P)$ be the canonical $*$-epimorphism.
We define the \emph{reduced quotient strong covariance algebra} $q_{\scv}(\ca_{\la}(P))$ be the quotient of $\ca_\la(P)$ by the ideal $\la \rho(\I_\infty)$.
\end{definition}

We therefore can update the previous diagrams to obtain
\[
\xymatrix{
[\ca_{\textup{iso}}(P)]_e \ar[r]^{\rho} \ar[dr] &
[\ca_s(P)]_e \ar[r]^{\la} \ar[d] &
[\ca_\la(P)]_e \ar[r]^{\Phi:= \bigoplus\limits_{\textup{fin } F \subseteq G} \Phi_F|_{X_F}} \ar[d]^{q_{\scv}} &
\prod\limits_{\text{fin } F \subseteq G} \B(X_F) \ar[d]^{q_{\qd}} \\
& 
[\ca_{\scv}(P)]_e \ar[r] & 
[q_{\scv}(\ca_{\la}(P))]_e \ar[r] &
\quo{\prod\limits_{\text{fin } F \subseteq G} \B(X_F)}{c_0(\B(X_F) \mid {\text{fin } F \subseteq G})} .
}
\]
Since $q_{\scv}(\ca_{\la}(P))$ is a quotient of $\ca_\la(P)$ by an induced ideal it inherits the coaction of $G$.
Hence $q_{\scv}$ is an equivariant unital $*$-representation of $\ca_{\scv}(P)$, and we derive the following corollary of Theorem \ref{T:inj ex}.

\begin{corollary}\label{C:cond exp inj}
Let $P$ be a unital subsemigroup in a group $G$.
Then $[q_{\scv}(\ca_{\la}(P))]_e \simeq [\ca_{\scv}(P)]_e$.
Moreover any unital $*$-representation of $q_{\scv}(\ca_{\la}(P))$ that admits a coaction of $G$ is faithful on $[q_{\scv}(\ca_{\la}(P))]_e$.
\end{corollary}

\subsection{Fell bundles}

If $P$ embeds in a group $G$, then there are canonical coactions of $G$ on the universal C*-algebras $\ca_s(P)$ and $\ca_{\scv}(P)$, which give gradings and thus Fell bundles.

\begin{definition}
Let $P$ be a unital semigroup in a group $G$ and consider the C*-algebras $\ca_s(P)$ and $\ca_{\scv}(P)$.
We define the Fell bundles $\P_s$ and $\P_{\scv}$ over $G$ by
\[
\P_{s, g} := [\ca_s(P)]_g
\qand
\P_{\scv, g} := [\ca_{\scv}(P)]_g
\quad
\text{ for all } g \in G.
\]
\end{definition}

Here we show the connections with the other algebras.

\begin{theorem} \label{T:Fell bdl univ}
Let $P$ be a unital semigroup in a group $G$.
Then
\[
\ca(\P_s) \simeq \ca_s(P) \simeq \ca(\I_V),
\quad
\ca_\la(\P_s) \simeq \ca_\la(\I_V),
\qand
\ca(\P_{\scv}) \simeq \ca_{\scv}(P),
\]
by canonical $*$-homomorphisms that fix $P$.
Moreover we have that $\J$ is independent if and only if $\ca_{\la}(\P_s) \simeq \ca_{\la}(P)$.

If $G$ is exact (but $\J$ is not necessarily independent) then
\[
\ca_{\la}(\P_{\scv}) \simeq q_{\scv}(\ca_{\la}(P))
\]
by a canonical $*$-homomorphism that fixes $P$.
\end{theorem}

\begin{proof}
We have already seen in Theorem \ref{T:inv sgrp} that $\ca_s(P) \simeq \ca(\I_V)$. To prove the isomorphism to the full C*-algebra of the bundle, recall that $\ca_s(P)$ is a quotient of $\ca_{\textup{iso}}(P)$ by an induced ideal (as the extra relations in $\ca_s(P)$ are $e$-graded).
Hence the representations of $\P_s$ are automatically representations of $\ca_s(P)$ and conversely.
A similar argument applies for $\P_{\scv}$ to prove $\ca(\P_{\scv}) \simeq \ca_{\scv}(P)$.

The proof of Theorem \ref{T:inv sgrp} asserts that a representation of $\I_V$ implements a representation of $\ca_s(P)$ and thus of $\P$.
Since the reduced representation of $\I_V$ has a faithful conditional expectation we derive by \cite{Exe97} that $\ca_\la(\I_V) \simeq \ca_\la(\P_s)$.

For the second part we have that $\P_{s,e} \simeq [\ca_s(P)] = D_s(P)$.
Therefore by Corollary \ref{C:nor-li} we have that $\J$ is independent if and only if $\P_{s,e} \simeq D_\la(P)$, equivalently if and only if $\ca_\la(P)$ is an injective representation of $\P_s$.
The latter is equivalent to $\ca_\la(\P_s) \simeq \ca_\la(P)$ by \cite{Exe97}, as $\ca_\la(P)$ has a faithful conditional expectation.

Suppose now that $G$ is exact.
By Corollary \ref{C:cond exp inj} $q_{\scv}(\ca_{\la}(P))$ is an isometric representation of $\P_{\scv}$.
Since $G$ is exact, the normal coaction of $G$ on $\ca_\la(P)$ is inherited by $q_{\scv}(\ca_{\la}(P))$.
Thus by \cite{Exe97} we have that $q_{\scv}(\ca_{\la}(P)) \simeq \ca_\la(\P_{\scv})$.
\end{proof}

\subsection{The partial crossed product realization}

A alternative description of $\ca_\la(P)$ as a reduced partial crossed product has been given in \cite[Proposition 3.10]{Li17}.
Let us recall some notation and facts about partial product systems from \cite{Exe17, McC95}.

Suppose that a group $G$ with identity $e$ acts on a topological space $X$ by a partial action $\theta$ in the sense that:
\begin{enumerate}
\item there is a collection $\{\Om_g\}_{g \in G}$ of open subsets of $X$ such that $\Om_e = X$;
\item there is a collection $\{\theta_g\}_{g \in G}$ of homeomorphisms $\theta_g \colon \Om_{g^{-1}} \to \Om_g$ such that $\theta_e = \id_X$;
\item for all $g_1, g_2 \in G$ we have that $\theta_{g_2} (\Om_{(g_1 g_2)^{-1}} \cap \Om_{g_2^{-1}}) = \Om_{g_1^{-1}} \cap \Om_{g_2}$, and $\theta_{g_1 g_2}(x) = \theta_{g_1} \circ \theta_{g_2}(x)$ for all $x \in \Om_{(g_1 g_2)^{-1}} \cap \Om_{g_2^{-1}}$.
\end{enumerate}
Then the reduced crossed product algebra $C_0(X) \rtimes_r G$ is defined in the following way.
On $\ell^2(G) \otimes \ell^2(X)$ let the twisted representation of $C_0(\Om_g)$ given by
\[
\mu(f) \de_h \otimes \xi = \de_h \otimes M_h(f) \xi,
\]
where
\[
(M_h(f) \xi)(x) = 
\begin{cases}
f(\theta_h(x)) \xi(x) & \text{ if } x \in \Om_{h^{-1}}, \\
0 & \text{ if } x \notin \Om_{h^{-1}}.
\end{cases}
\]
Furthermore, let $E_g$ be the projection on $\ol{\mu(C_0(\Om_{g^{-1}})) (\ell^2(G) \otimes \ell^2(X))}$ so that
\[
E_g(\de_h \otimes \de_x)
=
\begin{cases}
\de_h \otimes \de_x  & \text{ if } x \in \Om_{(gh)^{-1}} \cap \Om_{h^{-1}}, \\
0 & \text{ otherwise},
\end{cases}
\]
and for $g \in G$ define
\[
\dot{u}_g := (\la_g \otimes I_{\ell^2(X)}) \cdot E_g.
\]
Then the \emph{reduced crossed product} of the partial action is defined by 
\[
C_0(X) \rtimes_r G := \ol{\spn}\{ \mu(f_g) \dot{u}_g \mid f \in C_0(\Om_g), g \in G \},
\]
and it follows that it is a C*-algebra.
It is known that the unitary operator
\[
U \colon \ell^2(G) \otimes \ell^2(X) \otimes \ell^2(G) \to \ell^2(G) \otimes \ell^2(X) \otimes \ell^2(G) ;
\de_h \otimes \de_x \otimes \de_{h'} \mapsto \de_h \otimes \de_x \otimes \de_{h h'}
\]
induces a ``reduced'' coaction $\de_\la$
\[
\xymatrix{
C_0(X) \rtimes_r G \ar[rr]^{\simeq} & & (C_0(X) \rtimes_r G) \otimes I \ar[rr]^{\ad_{U}} & & (C_0(X) \rtimes_r G) \otimes \ca_\la(G),
}
\]
and thus a normal coaction $\de$ of $G$ on $C_0(X) \rtimes_r G$.

Recall that there is also the \emph{universal partial crossed product}, denoted by $C_0(X) \rtimes G$, which is the universal C*-algebra subject to $\ell^1$-bounded representations of the algebra of monomials $f_g \de_g$ where $f_g \in C_0(\Om_g)$, the $\de_g$ are invertible of norm one for $g \in G$, and the covariant multiplication is given by
\[
(f_1 \de_{g_1}) \cdot (f_2 \de_{g_2}) := \theta_{g_1}(\theta_{g_1^{-1}}(f_1) f_2) \de_{g_1 g_2}
\foral
f_i \in C_0(\Om_{g_i}), g_i \in G.
\]

A closed set $Y \subseteq X$ is called \emph{invariant} under $\theta$, or \emph{$G$-invariant}, if
\[
\theta_g(Y \cap \Om_{g^{-1}}) \subseteq Y
\foral
g \in G.
\]
In this case one can define a partial action on $\Om'_g := Y \cap \Om_g$ by the restrictions $\theta_g' := \theta_g|_{\Om_g'}$.
If the $\theta_g'$ are also homeomorphisms then we can define the reduced partial crossed product $C_0(Y) \rtimes_r G$, and there is a canonical $*$-homomorphism
\[
C_0(X) \rtimes_r G \to C_0(Y) \rtimes_r G ; \mu(f_g) \dot{u}_g \mapsto \mu(f_g|_{\Om'_g}) \dot{u}_g,
\]
in the sense that it sends generators to their restriction on $Y$.
This $*$-homomorphism is faithful if and only if $X = Y$.

Conversely, quotients of $C_0(X) \rtimes_r G$ produce $G$-invariant closed sets of $X$.
Towards this end let a $*$-homomorphism
\[
\Phi \colon C_0(X) \rtimes_r G \to \B(H),
\]
and let us write $\theta_g^*$ for the $*$-endomorphism induced by $\theta_{g^{-1}}$.
For convenience let us identify $C_0(\Om_g)$ with $\mu(C_0(\Om_g))$.
We will show that there exists a $*$-homomorphism
\[
\vartheta_g^* \colon \Phi(C_0(\Om_{g^{-1}})) \to \Phi(C_0(\Om_{g}))
\text{ such that }
\vartheta_g^* \circ \Phi|_{C_0(\Om_{g^{-1}})}  = \Phi \circ \theta_g^*.
\]
Then it will follow that $\Phi(C_0(X))$ is $G$-invariant, and by duality defines a closed $G$-invariant subspace $Y$ of $X$.
By Arveson's Extension Theorem there exists a unital completely positive map
\[
\phi \colon \B(\ell^2(G) \otimes \ell^2(X)) \to \B(H)
\text{ such that }
\phi|_{C_0(X) \rtimes_r G} = \Phi.
\]
We directly define $\vartheta_g^* \colon \Phi(C_0(\Om_{g^{-1}})) \to \B(H)$ by
\[
\vartheta_g^*(\Phi(f)) := \phi(\dot{u}_g) \Phi(f) \phi(\dot{u}_g)^*
\foral
f \in C_0(\Om_{g^{-1}}).
\]
It is clear that $\vartheta_g^*$ is a contractive completely positive map implemented by the contraction $\phi(\dot{u}_g)$.

\begin{proposition}\label{P:ind inv}
With the aforementioned notation, we have that
\[
\vartheta_g^* \circ \Phi|_{C_0(\Om_{g^{-1}})} = \Phi \circ \theta_g^*.
\]
Therefore $\vartheta_g^*$ is a $*$-isomorphism onto $\Phi(C_0(\Om_{g^{-1}}))$.
\end{proposition}

\begin{proof}
By construction we have that $C_0(X) \rtimes_r G$ is in the multiplicative domain of $\phi$ and therefore $\phi(\dot{u}_g x) =\phi(\dot{u}_g) \phi(x)$ for all $x \in C_0(X) \rtimes_r G$.
In particular for every $f \in C_0(\Om_{g^{-1}}) \simeq \mu(C_0(\Om_{g^{-1}}))$ with $f \geq 0$ we have that
\begin{align*}
 \vartheta_g^*(\Phi(f)) 
& = \phi(\dot{u}_g) \Phi(f) \phi(\dot{u}_g)^* 
= \phi(\dot{u}_g) \Phi(f^{\tfrac{1}{2}}) \Phi(f^{\tfrac{1}{2}}) \phi(\dot{u}_g)^* \\
& = \Phi(\dot{u}_g f^{\tfrac{1}{2}}) \Phi(f^{\tfrac{1}{2}} \dot{u}_g^*) 
= \Phi(\dot{u}_g f \dot{u}_g^*) = \Phi(\theta_g^*(f)),
\end{align*}
where we used that $\Phi(f^{\tfrac{1}{2}})$ lies in the multiplicative domain of $\phi$ for the third equality. 
This completes the proof since the maps involved above are linear and continuous.
\end{proof}

Next we pass to the connection of partial crossed product with our context.
Let $P$ be a unital subsemigroup in a group $G$ that it generates, and let the isometries $p \mapsto V_p$ be the left regular representation on $\ell^2(P)$.
On $\I_V^\times := \I_V \setminus \{0\}$ we define the homomorphism
\[
\si \colon I_V^\times \to G ; V^*_{p_1} V_{q_1} \cdots V^*_{p_n} V_{q_n} \mapsto p_1^{-1} q_1 \cdots p_n^{-1} q_n.
\]
Then $D_\la(P) = \ol{\spn}\{ \si^{-1}(e) \}$, and define
\[
D_{g^{-1}} := \ol{\spn} \{ V^*V \mid V \in I_V^\times, \si(V) = g\}.
\]
From \cite[Section 3.3]{Li17} every $D_{g^{-1}}$ is an ideal of the abelian C*-algebra $D_\la(P)$.
By identifying $\ell^2(P)$ with a subspace of $\ell^2(G)$ we may view $\ca_\la(P) \subseteq \B(\ell^2(G))$ and we get that
\[
V = \la_g V^* V \foral V \in \si^{-1}(g).
\]
Therefore there is an induced $*$-homomorphism implemented by $\la_g$, namely
\[
\theta_g^* \colon D_{g^{-1}} \to D_{g} ; V^*V \mapsto \la_g V^*V \la_g^* = VV^*.
\]
By construction $D_\la(P)$ and $D_{g^{-1}}$ are abelian and we denote their spectra by $\Om_P$ and $\Om_{g^{-1}}$, respectively.
Then the actions $\theta_g$ induced by $\theta_{g^{-1}}^*$ form a partial action of $G$ on $\Om_P$, i.e.,
\[
\theta_g^*(\chi) = \chi \circ \theta_{g^{-1}} \foral \chi \in D_{g^{-1}} \simeq C_0(\Om_{g^{-1}}).
\]

\begin{proposition} 
\cite[Proposition 3.10]{Li17}
Let $P$ be a unital semigroup in a group $G$.
Then there is a canonical $*$-isomorphism
\[
\ca_\la(P) \to D_\la(P) \rtimes_r G ; V_p \mapsto \mu(V_p V_p^*) \dot{u}_p = \dot{u}_p.
\]
\end{proposition}

In \cite{Li17} it was shown that there is a smallest non-empty closed $G$-invariant subspace $\partial \Om_P$ of $\Om_P$.
The subspace $\partial \Om_P$ can be identified by using the semilattice of idempotents in $P$.

\begin{definition}
Let $P$ be a unital semigroup in a group $G$.
Let $\partial \Om_P$ be the smallest non-empty closed $G$-invariant subspace of $\Om_P$.
The C*-algebra
\[
\partial \ca_\la(P) := C(\partial \Om_P) \rtimes_r G
\]
is called the \emph{boundary quotient} of $\ca_\la(P)$.
\end{definition}

The canonical $*$-epimorphisms
\[
\ca_s(P) \to \ca_\la(P) \to \partial \ca_\la(P)
\]
restricted to the diagonals produce the injections
\[
\partial \Om_P \hookrightarrow \Om_P \hookrightarrow {\rm Spec}(D_s(P)).
\]
Thus $\partial \Om_P$ is a minimal $G$-invariant subspace of ${\rm Spec}(D_s(P))$. 
It is not immediate but by \cite[Lemma 5.7.10]{CELY17} it follows that $\partial \Om_P$ is the smallest $G$-invariant subspace of ${\rm Spec}(D_s(P))$, as well. 
The partial crossed product picture connects with the previous C*-algebras in the following way.

\begin{theorem} \label{T:un pcp}
Let $P$ be a unital semigroup in a group $G$.
Then 
\[
\ca_s(P) \simeq D_s(P) \rtimes G
\qand
\ca_{\scv}(P) \simeq C(\partial \Om_P) \rtimes G,
\]
and therefore
\[
\ca_{\la}(\P_{\scv}) \simeq C(\partial \Om_P) \rtimes_r G \equiv \partial \ca_\la(P),
\]
by canonical $*$-homomorphisms that fix $P$.
\end{theorem}

\begin{proof}
The first part follows directly as in the reduced case \cite[Proposition 3.10]{Li17}, where now the $*$-isomorphism is implemented by the universal properties.
Since $\ca_{\scv}(P)$ is a quotient of $\ca_s(P)$ by an induced ideal we then get that
\[
\ca_{\scv}(P) \simeq C(Y) \rtimes G
\qfor 
C(Y) : = [\ca_{\scv}(P)]_e.
\]
It is clear that $C(\partial \Om_P) \rtimes G$ inherits a coaction of $G$ from $D_s(P) \rtimes G$ due to exactness of the universal construction \cite{Exe17}.
By definition there exists a unital $G$-equivariant $*$-epimorphism
\[
\ca_{\scv}(P) \simeq C(Y) \rtimes G \to C(\partial \Om_P) \rtimes G.
\]
Then Theorem \ref{T:inj ex} yields that the map is $*$-isomorphic on the fixed point algebra and thus
\[
C(Y) \simeq [\ca_{\scv}(P)]_e \simeq [C(\partial \Om_P) \rtimes G]_e \simeq C(\partial \Om_P).
\]
Therefore we get the required $Y \simeq \partial \Om_P$, and thus
\[
\ca(\P_{\scv}) \simeq \ca_{\scv}(P) \simeq C(\partial \Om_P) \rtimes G.
\]
Recall here that $\ca_{\scv}(P) \simeq \ca(\P_{\scv})$ from Theorem \ref{T:Fell bdl univ}.
Since $\P_{\scv}$ is the Fell bundle of a partial crossed product \cite[Proposition 16.28]{Exe17}, we then get that its reduced C*-algebra passes down to the reduced partial crossed product giving
\[
\ca_\la(\P_{\scv}) \simeq C(\partial \Om_P) \rtimes_r G,
\]
and the proof is complete.
\end{proof}

\begin{remark}
In \cite{LS21} Laca and Sehnem identify the exact relations that define the Fell bundle of the canonical coaction of $G$ on $\ca_\la(P)$.
These come from $*$-representations $\Phi$ of $\ca_s(P)$ that satisfy the additional axiom that $\prod_{\By \in F} (\Phi(e_{\Bx}) - \Phi(e_{\By})) = 0$ whenever $F \subseteq \J$ is finite and $\Bx = \cup_{\By \in F} \By$.
The quotient of $\ca_s(P)$ determined by this set of extra relations in the context of \cite{LS21} is called \emph{the universal Toeplitz algebra of $P$}, and is denoted by $\T_u(P)$.
The quotient map $\ca_s(P) \to \ca_{\scv}(P)$ from Proposition \ref{P:scv ex} can also be seen by combining \cite[Proposition 3.22]{LS21} with the isomorphism of items (1) and (2) in \cite[Theorem 6.13]{LS21}.
The isomorphism $\ca_{\scv}(P) \simeq C(\partial \Om_P) \rtimes G$ from Theorem \ref{T:un pcp} also appears in \cite[Theorem~6.13]{LS21} as the partial crossed product realization of the covariance algebra of the product system with one-dimensional fibers over $P$, which is viewed there as the natural full boundary quotient of $\T_u(P)$.
\end{remark}

\section{Noncommutative boundaries and co-boundaries}\label{S:nc bdy}

We will show that $\partial \ca_\la(P)$ carries two types of co-universality.
The first one is in the C*-context with respect to the representations of $\ca_s(P)$, while the second one
is in the nonselfadjoint context, with respect to contractive representations of the tensor algebra $\A(P)$ that are compatible with the canonical coaction of $G$.
We introduce some terminology to make this precise.

\begin{definition}
Let $P$ be a unital semigroup in a group $G$.
An isometric semigroup representation $T$ of $P$ will be called \emph{constructible} if it induces a representation of $\ca_s(P)$.

We denote by $\partial_G [\ca_s(P)]$ the C*-algebra generated by a $G$-equivariant constructible isometric representation of $P$ with the following co-universal property:
for every $G$-equivariant constructible isometric representation $T\colon P \to \B(H)$ there exists a canonical $*$-epimorphism $\ca(T) \to \partial_G [\ca_s(P)]$.
\end{definition}

We now prove the first co-universal property of $\ca_s(P)$.

\begin{theorem}\label{T:co-universal}
Let $P$ be a unital semigroup in a group $G$.
Then
\[
\partial \ca_\la(P) \simeq \partial_G [\ca_s(P)].
\]
\end{theorem}

\begin{proof}
Let $T$ be an non-trivial $G$-equivariant representation of $\ca_s(P) \simeq D_s(P) \rtimes G$.
Let $\T = \{\ca(T)_g\}_{g \in G}$ be the associated Fell bundle on $\ca(T)$.
Since $T$ is non-trivial, the isometry $T_p$ is non-zero for every $p \in P$, so the Fell bundle $\T$ is non-zero and thus $\ca_\la(\T)$ is non-zero.
By \cite[Proposition 21.3]{Exe17} there is a canonical $*$-epimorphism
\[
D_s(P) \rtimes_r G \simeq \ca_\la(\P_s) \to \ca_\la(\T).
\]
By Proposition \ref{P:ind inv} we have that $[\ca(T)]_e = C(Y)$ for some $G$-invariant closed subspace of the spectrum of $D_s(P)$.
Therefore, due to the existence of the faithful conditional expectation on $\ca_\la(\T)$ we can write
\[
\ca_\la(\T) \simeq C(Y) \rtimes_r G.
\]
By \cite[Lemma 5.7.10]{CELY17} we have the inclusion $\partial \Om \hookrightarrow Y$.
We deduce that
\[
\ca(T) \to \ca_\la(\T) \simeq C(Y) \rtimes_r G \to C(\partial \Om) \rtimes_r G,
\]
and the proof is complete.
\end{proof}

The next aim is to show that $\partial \ca_\la(P)$ is the C*-envelope of the cosystem $(\A(P), G, \ol{\de})$.
First we show that it is a boundary quotient of $\A(P)$, and hence the terminology makes sense both in the selfadjoint and the nonselfadjoint context.

\begin{proposition}\label{P:boundary}
Let $P$ be a unital semigroup in a group $G$ and $Y$ be a closed $G$-invariant subspace of $\Om_P$.
Then the canonical embedding
\[
\A(P) \hookrightarrow C(\Om_P) \rtimes_r G \to C(Y) \rtimes_r G
\]
is unital completely isometric.
Thus $C(Y) \rtimes_r G$ becomes a (normal) C*-cover for $(\A(P), G, \ol{\de})$.
\end{proposition}

\begin{proof}
Let us denote by $\iota \colon \A(P) \to C(Y) \rtimes_r G$ the canonical completely contractive embedding, where $C(Y) \rtimes_r G$ acts on $\ell^2(G) \otimes \ell^2(Y)$.
By construction
\[
\iota(V_p) = (\la_p \otimes I) \cdot E_p
\foral 
p \in P,
\]
where $E_p$ is the projection on 
\[
E_p(\ell^2(G) \otimes \ell^2(Y)) = \sumoplus_{h \in G} \bC\de_h \otimes \ell^2(\Om_{(p h)^{-1}} \cap \Om_{h^{-1}} \cap Y).
\]
Notice that for $r \in P$ we have that $I = V_r^* V_r \in D_{r^{-1}}$ and so $\Om_{r^{-1}} = \Om_P$.
Hence
\[
\Om_{(p h)^{-1}} \cap \Om_{h^{-1}} \cap Y = Y \foral h \in P.
\]
Thus, if $Q_P \colon \ell^2(G) \to \ell^2(P)$ is the canonical projection we see that
\[
E_p \cdot (Q_P \otimes I) = (Q_P \otimes I) \cdot E_p = (Q_P \otimes I)
\]
for all $p \in P$.
Therefore
\[
\iota(V_p) \cdot (Q_P \otimes I)
=
(\la_p \otimes I) \cdot (Q_P \otimes I)
=
V_p \otimes I
=
(Q_P \otimes I) \cdot (\la_p \otimes I) \cdot (Q_P \otimes I),
\]
and consequently
\[
\iota(a) \cdot (Q_P \otimes I) = a \otimes I
\foral
a \in \A(P).
\]
Let $n \in \bN$ and $[a_{ij}] \in \M_n \otimes \A(P)$.
Then we have that
\[
\| [a_{ij}] \|
\geq
\| [\iota(a_{ij})] \|
\geq
\| [\iota(a_{ij})] \cdot (I_n \otimes Q_P \otimes I) \|
=
\| [a_{ij} \otimes I] \|
=
\|[a_{ij}]\|,
\]
giving that the map $\id_n \otimes \iota$ is isometric.
As this holds for any $n \in \bN$ it follows that $\iota$ is completely isometric.

Finally recall that $C(Y) \rtimes_r G$ admits a normal coaction $\de'$ of $G$.
By definition we have that
\[
(\iota \otimes \id) \ol{\de}(V_p) = \dot{u}_p \otimes u_p = ((\la_p \otimes I) \cdot E_p) \otimes u_p = \de' \iota(V_p)
\]
for all $p \in P$, and the proof is complete.
\end{proof}

We can now prove the second co-universal result for $\partial \ca_\la(P)$.

\begin{theorem}\label{T:part}
Let $P$ be a unital semigroup in a group $G$.
Then
\[
\partial \ca_\la(P) \simeq \cenv(\A(P), G, \ol{\de}).
\]
\end{theorem}

\begin{proof}
By Theorem \ref{T:DKKLL}, normality of the coaction on $\A(P)$ implements a normal coaction on $\cenv(\A(P), G, \ol{\de})$.
Since any $G$-equivariant quotient of $\ca_\la(P)$ implements a partial action by Proposition \ref{P:ind inv}, faithfulness of the conditional expectation on $\cenv(\A(P), G, \ol{\de})$ yields
\[
\cenv(\A(P), G, \ol{\de}) \simeq C(Y) \rtimes_r G
\qfor
C(Y) := [\cenv(\A(P), G, \ol{\de})]_e.
\]
On the other hand by Proposition \ref{P:boundary} we have that $\partial \ca_\la(P)$ is a C*-cover for $(\A(P), G, \ol{\de})$, and so there exists a canonical $*$-epimorphism
\[
\Phi \colon \partial \ca_\la(P) = C(\partial \Om_P) \rtimes_r G \to \cenv(\A(P), G, \ol{\de})
\]
that intertwines the conditional expectations. 
Hence $\Phi$ restricts to a surjection from $C(\partial \Om_P)$ onto $C(Y)$.
Minimality of $\partial \Om_P$ yields $\partial \Om_P = Y$, and thus $\Phi$ is faithful.
\end{proof}

Recall that the action of $G$ on a partial system $(\{\Om_g\}_{g \in G}, \{\theta_g\}_{g \in G})$ is called \emph{topologically free} if for every $e \neq g \in G$ the set $\{\om \in \Om_{g^{-1}} \mid \theta_g(\om) \neq \om\}$ is dense in $\Om_{g^{-1}}$.
Next we see that under a topological freeness assumption we have that the \v{S}ilov boundary of the cosystem is the usual \v{S}ilov boundary of $\A(P)$.

\begin{theorem}\label{T:top free}
Let $P$ be a unital semigroup in a group $G$.
If the partial action of $G$ on $\partial \Om_P$ is topologically free then
\[
\partial \ca_\la(P) \simeq \cenv(\A(P)).
\]
\end{theorem}

\begin{proof}
Clearly any C*-cover of the cosystem is a C*-cover for the algebra.
Thus by Theorem \ref{T:part} and the definition of $\cenv(\A(P))$ we have a canonical $*$-epimorphism 
\[
\partial \ca_\la(P) \simeq \cenv(\A(P), G, \ol{\de}) \to \cenv(\A(P)).
\]
By \cite[Corollary 3.22]{Li17} we get that $\partial \ca_\la(P)$ is a simple C*-algebra and thus the $*$-epimorphism is faithful.
\end{proof}

Let us now consider the Ore semigroup case.
Recall that $P$ is a \emph{left-Ore semigroup} if it is cancellative and left-reversible, in the sense that it satisfies $p P \cap q P \neq \mt$ for all $p, q \in P$. 
A semigroup $P$ is  left-Ore  if and only if it embeds in a group $G$ in such a way that $G = P P^{-1}$. Equivalently, the set $P P^{-1}$ of formal right quotients forms a group.
See \cite{Dub43, Ore31}.
Modulo an easy translation from right to left reversibility, Laca \cite[Theorem 1.2]{Lac00} shows that a left-Ore semigroup has the following (universal) extension property for semigroup homomorphisms into groups: every semigroup homomorphism $P \to \mathcal G$ into a group $\G$ has a unique extension to a group homomorphism $G = PP^{-1} \to \G$.
See also \cite[Theorem 2.2.4]{DFK17} for an alternative proof that uses a construction of $G$ by direct limits.
When $P$ is an Ore semigroup then $\partial \Om_P$ is a singleton, hence $\partial \ca_\la(P) \simeq \{ \text{pt} \} \rtimes_r G = \ca_\la(G)$; see the comments following \cite[Definition 5.7.9]{CELY17}.
We will see that if $P$ is left-Ore and $G = P  P^{-1}$ is its enveloping semigroup, then $\cenv(\A(P)) = \ca_\la(G)$.
In fact, we will prove that this is a characteristic property of Ore semigroups.

\begin{theorem} \label{T:Ore}
Let $P$ be a unital semigroup in a group $G$ that it generates.
The following are equivalent:
\begin{enumerate}
\item $P$ is an Ore semigroup (in which case necessarily $G = P P^{-1}$).
\item The map $V_p \mapsto \la_p$ extends to a completely isometric map $\A(P) \to \ca_\la(G)$.
\item $\cenv(\A(P), G, \ol{\de}) \simeq \ca_\la(G)$ by a canonical $*$-homomorphism that fixes $P$.
\item $\cenv(\A(P)) \simeq \ca_\la(G)$ by a canonical $*$-homomorphism that fixes $P$.
\end{enumerate}
If any of the above holds then $\A(P)$ is hyperrigid.
\end{theorem}

\begin{proof}
We will show [(i) $\Leftrightarrow$ (ii) $\Rightarrow$ (iii) $\Rightarrow$ (iv) $\Rightarrow$ (ii)], and that item (ii) implies hyperrigidity of $\A(P)$.

\noindent
[(i) $\Rightarrow$ (ii)]:
If $P$ is an Ore semigroup then by construction $\partial \Om_P = \{ \text{pt} \}$ is a singleton; see the comments following \cite[Definition 5.7.9]{CELY17}.
Thus $\{ \text{pt} \} \rtimes_r G = \ca_\la(G)$ and the canonical embedding
\[
\A(P) \to \ca_\la(G) ; V_p \mapsto \la_p
\]
is completely isometric by Proposition \ref{P:boundary}.

\noindent
[(ii) $\Rightarrow$ (i)]:
Suppose there is a completely isometric map $\Phi \colon \A(P) \to \ca_\la(G)$ with $\Phi(V_p) = \la_p$ for all $p \in P$.
Then $\ca_\la(G)$ is a C*-cover for $\A(P)$.
Since contractive dilations of unitaries are trivial we get that $\ca_\la(G)$ is the C*-envelope.
For the same reason $\A(P)$ is hyperrigid.

Thus $\Phi$ extends uniquely to a $*$-epimorphism $\ca_\la(P) \to \ca_\la(G)$ which we denote by the same symbol.
Let $p, q \in P$.
Since
\[
\Phi(V_p V_p^* V_q V_q^*) = \la_p \la_p^* \la_q \la_q^* = I,
\]
we have that $V_p V_p^* V_q V_q^* \neq 0$.
It follows that $p P \cap q P \neq \mt$, and thus $P$ satisfies the Ore property.

\noindent
[(ii) $\Rightarrow$ (iii)]:
Assuming item (ii) we have that $\ca_\la(G)$ is a C*-cover of the cosystem $(\A(P), G, \ol{\de})$, and thus there is a canonical $*$-epimorphism
\[
\{ \text{pt} \} \rtimes_r G = \ca_\la(G) \to \cenv(\A(P), G, \ol{\de})
\]
that fixes $\A(P)$ and intertwines the faithful conditional expectations.
As the fixed point algebra of $\ca_\la(G)$ is trivially $\bC$ and the $*$-epimorphism is non-zero, we have that the $*$-epimorphism is faithful.

\noindent
[(iii) $\Rightarrow$ (iv)]:
Assuming item (iii) we have that $\ca_\la(G)$ is a C*-cover for $\A(P)$ by unitaries, and thus (as above) we deduce that it is its C*-envelope.

\noindent
[(iv) $\Rightarrow$ (ii)]:
Assuming item (iv) we have trivially that the map
\[
\A(P) \hookrightarrow \cenv(\A(P)) \simeq \ca_\la(G) ; V_p \mapsto \la_p
\]
is completely isometric, and the proof is complete.
\end{proof}

\begin{remark}\label{R:Ore}
We can combine Theorem \ref{T:co-universal} with Theorem \ref{T:Ore} to deduce that any item of Theorem \ref{T:Ore} is equivalent to having an equivariant $*$-isomorphism $\ca_\la(G) \simeq \partial_G [\ca_s(P)]$.
\end{remark}

\begin{remark}\label{rem:finalremark}
It is interesting to note that Theorem \ref{T:top free} and Theorem \ref{T:Ore}(iv) give the same conclusion $\partial \ca_\la(P) \simeq \cenv(\A(P))$  in opposite situations; the former applies to semigroups with a boundary spectrum that is large enough to support a topologically free partial action of $G$, while the latter applies when the boundary spectrum is a singleton so the boundary action is as far from topologically free as possible. 
Combined, they cover the following classes of examples:
\begin{enumerate}
\item abelian submonoids of groups, in particular, 
multiplicative semigroups of nonzero algebraic integers  (obviously Ore);
\item total orders in groups (obviously Ore);
\item Artin monoids in Artin groups of finite type (Ore by \cite{bri-sai});
\item quasi-lattice orders with trivial core (where the action is topologically free by \cite[Proposition 5.5]{CL07}), in particular, the subclass consisting of Artin monoids in right-angled Artin groups having no direct $\bZ$ factor, \cite[Corollary 5.7]{CL07};
\item the Thompson monoid $F^+$ in the Thompson group $F$ (Ore, in fact, lattice ordered);
\item the $ax+b$-semigroup $R \rtimes R^\times$ over an integral domain that is not a field (where the action is topologically free by combining \cite[Corollary 8]{Li10} with \cite[page 243]{CELY17}, or by \cite[Corollary 8.4 and Proposition 6.18]{LS21}).
\end{enumerate}
In fact only simplicity of the boundary quotient $\partial \ca_\la(P)$ is required in the proof of Theorem \ref{T:top free} to deduce that it coincides $\cenv(\A(P))$. 
A characterization of simplicity for general ring C*-algebras has been identified in \cite[Corollary 8]{Li10}.
\end{remark}



\begin{thebibliography}{99}

\bibitem{Arv69} W.B. Arveson,
\textit{Subalgebras of C*-algebras},
Acta Math.\ \textbf{123} (1969), 141--224.


\bibitem{bri-sai} E. Brieskorn and K. Saito,
\textit{Artin-{G}ruppen und {C}oxeter-{G}ruppen},
Invent.\ Math.\ \textbf{17} (1972), 245--271.

\bibitem{CLSV11} T.M. Carlsen, N.S. Larsen, A. Sims and S.T. Vittadello,
\textit{Co-universal algebras associated to product systems, and gauge-invariant uniquenss theorems},
Proc.\ Lond.\ Math.\ Soc.\ (3) \textbf{103} (2011), no.\ 4, 563--600.


\bibitem{Cob67} L.A. Coburn, 
\textit{The C*-algebra generated by an isometry}, 
Bull.\ Amer.\ Math.\ Soc.\ \textbf{73} (1967), 722--726.


\bibitem{CL02} J. Crisp and M. Laca, 
\textit{On the Toeplitz algebras of right-angled and finite-type Artin groups}, 
J.\ Austr.\ Math.\ Soc.\ \textbf{72} (2002), no.\ 2, 223--245.

\bibitem{CL07} J. Crisp and M. Laca, 
\textit{Boundary quotients of Toeplitz algebras of right-angled Artin groups}, 
J.\ Funct.\ Anal.\ \textbf{242} (2007), no.\ 1, 125--156.

\bibitem{CDL13} J. Cuntz, C. Deninger and M. Laca, 
\textit{C*-algebras of Toeplitz type associated with algebraic number fields}, 
Math.\ Ann.\ \textbf{355} (2013), no.\ 4, 1383--1423.

\bibitem{CEL13} J. Cuntz, S. Echterhoff and X. Li, 
\textit{On the K-theory of crossed products by automorphic semigroup actions}, 
Quart.\ J.\ Math.\ \textbf{64} (2013), no.\ 3, 747--784.

\bibitem{CEL15} J. Cuntz, S. Echterhoff and X. Li, 
\textit{On the K-theory of the C*-algebra generated by the left regular representation of an Ore semigroup}, 
J.\ Eur.\ Math.\ Soc.\ \textbf{17} (2015), no.\ 3, 645--687.

\bibitem{CELY17} J. Cuntz, S. Echterhoff, X. Li and G. Yu, 
\textit{K-Theory for group C*-Algebras and semigroup C*-algebras}, 
Oberwolfach Seminars, 47, Birkh{\"a}user/Springer, Cham, 2017.

\bibitem{DFK17} K.R. Davidson, A.H. Fuller and E.T.A. Kakariadis,
\textit{Semicrossed products of operator algebras by semigroups},
Mem.\ Amer.\ Math.\ Soc.\ \textbf{247} (2017), no.\ 1168, v+97 pp.

\bibitem{DKKLL20} A. Dor-On, E.T.A. Kakariadis, E.G. Katsoulis, M. Laca and X. Li, 
\textit{C*-envelopes of operator algebras with a coaction and co-universal C*-algebras for product systems}, 
preprint.

\bibitem{DK20} A. Dor-On and E.G. Katsoulis, 
\textit{Tensor algebras of product systems and their C*-envelopes}, 
J.\ Funct.\ Anal.\ \textbf{278} (2020), no.\ 7, 108416, 32 pp.

\bibitem{Dou72} R. G. {Douglas},
\textit{On the C*-algebra of a one-parameter semigroup of isometries}, 
Acta Math.\ \textit{128} (1972), no.\ 3--4, 143--151.

\bibitem{DM05} M.A. Dritschel and S.A. McCullough,
\textit{Boundary representations for families of representations of operator algebras and spaces}, 
J.\ Operator Theory \textbf{53} (2005), no.\ 1, 159--167.

\bibitem{Dub43} P. Dubreil, 
\textit{Sur les probl\`{e}mes d'immersion et la th\'{e}orie des modules},
C.\ R.\ Acad.\ Sci.\ Paris \textbf{216} (1943), 625--627.


\bibitem{Exe97} R. Exel,
\textit{Amenability of Fell bundles},
J.\ Reine Angew.\ Math.\ \textbf{492} (1997), 41--73.

\bibitem{Exe17} R. Exel,
\textit{Partial dynamical systems, Fell bundles and applications},
volume \textbf{224} of \textit{Mathematical Surveys and Monographs, American Mathematical Society}.
American Mathematical Society, Providence, RI, 2017. vi+321 pp.

\bibitem{Ham79} M. Hamana,
\textit{Injective envelopes of operator systems}, 
Publ.\ Res.\ Inst.\ Math.\ Sci.\ \textbf{15} (1979), no.\ 3, 773--785.

\bibitem{Kak19} E.T.A. Kakariadis,
\textit{Finite dimensional approximations for Nica-Pimsner algebras},
Ergodic Theory Dynam.\ Systems \textbf{40} (2020), no.\ 12, 3375-3402.


\bibitem{Lac00} M. Laca, 
\textit{From endomorphisms to automorphisms and back: dilations and full corners}, 
J.\ London Math.\ Soc.\ (2) \textbf{61} (2000), no.\ 3, 893--904.

\bibitem{LS21} M. Laca, C. Sehnem,
\textit{Toeplitz algebras of semigroups}, 
preprint (preprint at arXiv:2101.06822).

\bibitem{LOS18} X. Li, T. Omland and J. Spielberg, 
\textit{C*-algebras of right LCM one-relator monoids and Artin-Tits monoids of finite type}, 
Commun.\ Math.\ Phys.\ (2020), doi.org/10.1007/200220-020-03758-5.
(preprint at arXiv:1807.08288.)

\bibitem{Li10} X. Li, 
\textit{Ring C*-algebras}, 
Math.\ Ann.\ \textbf{348} (2010), no.\ 4, 859--898.

\bibitem{Li12} X. Li, 
\textit{Semigroup C*-algebras and amenability of semigroups}, 
J.\ Funct.\ Anal.\ \textbf{262} (2012), no.\ 10, 4302--4340. 

\bibitem{Li13} X. Li, 
\textit{Nuclearity of semigroup C*-algebras and the connection to amenability}, 
Adv.\ Math.\ \emph{244} (2013), 626--662.

\bibitem{Li17} X. Li,
\textit{Partial transformation groupoids attached to graphs and semigroups},  
Int.\ Math.\ Res.\ Not.\ IMRN \textbf{2017}, no.\ 17, 5233--5259. 

\bibitem{McC95} K. McClanahan,
\emph{K-theory for partial crossed products by discrete groups},
J.\ Funct.\ Anal.\ \textbf{130} (1995), no.\ 1, 77--117.

\bibitem{Mur87} G.J. {Murphy}, 
\textit{Ordered groups and Toeplitz algebras}, 
J.\ Operator Theory \textbf{18} (1987), no.\ 2, 303--326. 

\bibitem{Nic92} A. Nica,
\textit{C*-algebras generated by isometries and Wiener-Hopf operators},
J.\ Operator Theory \textbf{27} (1992), no.\ 1, 17--52.

\bibitem{Nor14} M.D. Norling, 
\textit{Inverse semigroup C*-algebras associated with left cancellative semigroups}, 
Proc.\ Edinb.\ Math.\ Soc.\ (2) \textbf{57} (2014), no.\ 2, 533--564.

\bibitem{Ore31} O. Ore, 
\textit{Linear equations in non-commutative fields}, 
Ann.\ Math.\ (2) \textbf{32} (1931), no.\ 3, 463--477.

\bibitem{Pau02} V.I. Paulsen,
\textit{Completely bounded maps and operator algebras}, 
volume \textbf{78} of \textit{Cambridge Studies in Advanced Mathematics}.
Cambridge University Press, Cambridge, 2002. xii+300 pp.

\bibitem{Qui96} J. Quigg, 
\textit{Discrete C*-coactions and C*-algebraic bundles}, 
J.\ Austral.\ Math.\ Soc.\ Ser.\ A \textbf{60} (1996), no.\ 2, 204--221.


\bibitem{Seh18} C.F. Sehnem, 
\textit{On C*-algebras associated to product systems}, 
J.\ Funct.\ Anal.\ \textbf{277} (2019), no.\ 2, 558--593.

\end{thebibliography}
\end{document}